\documentclass[12pt, reqno]{amsart}
\usepackage[T1]{fontenc}
\usepackage[utf8]{inputenc}
\usepackage[english]{babel}
\usepackage{amsmath,amsfonts,amssymb,amsthm,mathrsfs}
\usepackage{accents,mathtools,extarrows}
\usepackage{yhmath}
\usepackage{graphicx}
\usepackage{rotating}
\usepackage{wasysym}
\usepackage{textcomp}
\usepackage{hyperref,cleveref}
\usepackage{xfrac}
\usepackage{enumitem}
\usepackage{tikz-cd}									
\usetikzlibrary{matrix,arrows,calc,positioning,decorations.pathmorphing}
\tikzset{
  shift left/.style ={commutative diagrams/shift left={#1}},
  shift right/.style={commutative diagrams/shift right={#1}}
}

\theoremstyle{plain}
\newtheorem{theorem}{Theorem}[section]
\newtheorem{lemma}[theorem]{Lemma}
\newtheorem{proposition}[theorem]{Proposition}
\newtheorem{corollary}[theorem]{Corollary}

\theoremstyle{definition}
\newtheorem{definition}[theorem]{Definition}
\newtheorem{example}[theorem]{Example}
\newtheorem{remark}[theorem]{Remark}
\newtheorem{construction}[theorem]{Construction}
\newtheorem{observation}[theorem]{Observation}
\newtheorem{preamble}[theorem]{Preamble}

\newenvironment{xreftheorem}[1]{\def\thexref{\ref{#1}}
\begin{thexreftheorem}\itshape}{\end{thexreftheorem}}
\newtheorem*{thexreftheorem}{Theorem \thexref}

\newenvironment{xrefcorollary}[1]{\def\thexcoro{\ref{#1}}
\begin{thexrefcorollary}\itshape}{\end{thexrefcorollary}}
\newtheorem*{thexrefcorollary}{Corollary \thexcoro}

\newcommand{\exend}{\unskip\nobreak\hfill$\circ$}
\newcommand{\defend}{\unskip\nobreak\hfill$\triangleleft$}

\newcommand{\Z}{\mathbb Z}
\newcommand{\Q}{\mathbb Q}
\newcommand{\R}{\mathbb R}
\newcommand{\C}{\mathbb C}
 
\newcommand{\F}{\mathbb F}

\renewcommand{\epsilon}{\varepsilon}
\renewcommand{\phi}{\varphi}
\renewcommand{\tilde}{\widetilde}

\newcommand{\exit}{\Pi{}^{\operatorname{exit}}}

\newcommand{\XBS}{\overline{X}{}^{BS}}
\newcommand{\XRBS}{\overline{X}{}^{RBS}}

\begin{document}

\title[Stratified homotopy type and reductive Borel--Serre]{The stratified homotopy type of the reductive Borel--Serre compactification}
\author{Mikala Ørsnes Jansen}
\address{Department of Mathematical Sciences, University of Copenhagen, 2100 Copenhagen, Denmark.}\email{mikala@math.ku.dk}
\thanks{The author was supported by the European Research Council (ERC) under the European Unions Horizon 2020 research and innovation programme (grant agreement No 682922), and the Danish National Research Foundation through the Centre for Symmetry and Deformation (DNRF92) and the Copenhagen Centre for Geometry and Topology (DNRF151).}

\begin{abstract}
We identify the exit path $\infty$-category of the reductive Borel--Serre compactification as the nerve of a $1$-category defined purely in terms of rational parabolic subgroups and their unipotent radicals. As immediate consequences, we identify the fundamental group of the reductive Borel--Serre compactification, recovering a result of Ji--Murty--Saper--Scherk, and we obtain a combinatorial incarnation of constructible complexes of sheaves on the reductive Borel--Serre compactification as elements in a derived functor category.
\end{abstract}

\maketitle

\tableofcontents

\section{Introduction}

\textbf{Background.}
Let $\mathbf{G}$ be a connected reductive linear algebraic group defined over $\Q$ whose centre is anisotropic over $\Q$. Let $\Gamma\leq \mathbf{G}(\Q)$ be a neat arithmetic group and consider the symmetric space $X$ of maximal compact subgroups of $\mathbf{G}(\R)$ on which $\Gamma$ acts by conjugation. The locally symmetric space $\Gamma\backslash X$ is a smooth manifold and a model for the classifying space of $\Gamma$; it is compact if and only if the $\Q$-rank of $\mathbf{G}$ is zero. The problem of compactifying such locally symmetric spaces has given rise to a number of different compactifications well suited for different purposes. In this paper we study the Borel--Serre and reductive Borel--Serre compactifications. 

\medskip

The \textit{Borel--Serre compactification} $\Gamma\backslash \XBS$, introduced in 1973 by Borel and Serre, is a compactification of $\Gamma\backslash X$ with the same homotopy type (\cite{BorelSerre}). This construction enabled Borel to calculate the rank of the K-groups $K_i(O_K)$, where $O_K$ is the ring of integers in a number field $K$ (\cite{Borel74}). Quillen also used the Borel--Serre compactification to show that these same K-groups, $K_i(O_K)$, are finitely generated (\cite{Quillen73a}).

\medskip

The \textit{reductive Borel--Serre compactification} $\Gamma\backslash \XRBS$ was introduced by Zucker in 1982 to facilitate the study of $L^2$-cohomology of $\Gamma\backslash X$ (\cite{Zucker82}, see also \cite{GoreskyHarderMacPherson}). It is a quotient of the Borel--Serre compactification, $\Gamma\backslash \XBS\rightarrow \Gamma \backslash \XRBS$, and it remedies the failure of $\Gamma\backslash \XBS$ to support $L^2$-partitions of unity. The reductive Borel--Serre compactification has been studied extensively and has come to play a central role in the theory of compactifications. It is well-suited for studying the $L^p$-cohomology of $\Gamma\backslash X$ (\cite{Zucker01}), it dominates all Satake compactifications (\cite[III.15.2]{BorelJi}), and it plays an important role in parametrising the continuous spectrum of $\Gamma\backslash X$ (\cite{JiMacPherson}). It is used to define weighted cohomology (\cite{GoreskyHarderMacPherson}) which is the main ingredient in the topological trace formula (\cite{GoreskyMacPherson92,GoreskyMacPherson03}) exploiting the fact that $L^2$-cohomology can be expressed locally on the reductive Borel--Serre compactification. It has moreover motivated the theory of $\mathscr{L}$-modules (\cite{Saper05a,Saper05b}) which is used to prove a conjecture of Rapoport (\cite{Rapoport}) and Goresky--MacPherson (\cite{GoreskyMacPherson88}) relating the intersection cohomology of the reductive Borel--Serre compactification with that of certain Satake compactifications.

\medskip

We will study these spaces as stratified topological spaces and determine their stratified homotopy type, or more precisely their exit path $\infty$-categories. The Borel--Serre compactification $\Gamma\backslash \XBS$ is naturally stratified as a manifold with corners over the poset of $\Gamma$-conjugacy classes of rational parabolic subgroups of $\mathbf{G}$. This stratification descends along the quotient map $\Gamma\backslash \XBS\rightarrow \Gamma\backslash \XRBS$, equipping the reductive Borel--Serre compactification with a natural stratification.

\medskip

Given a sufficiently nice stratified space $X$, one can define its \textit{exit path $\infty$-category} $\exit_\infty (X)$ (\cite[Definition A.6.2 and Theorem A.6.4]{LurieHA}); this is a natural analogue of the fundamental $\infty$-groupoid for topological spaces. Intuitively, the exit path $\infty$-category has as objects the points of the stratified space and as morphisms the paths which can only move ``upwards'' in the stratification, i.e. if $X_i\subset \overline{X_j}$ for two distinct strata $X_i, X_j\subset X$, then a path can move from $X_i$ to $X_j$, but not the other way. The higher simplices are stratum preserving homotopies between such ``exit paths'', and stratum preserving homotopies between such homotopies etc.

\medskip

The most important feature of the exit path $\infty$-category is that it classifies constructible sheaves, that is, sheaves which are locally constant on each stratum. This generalises the classical result that for a sufficiently nice topological space $X$, the monodromy functor gives an equivalence between representations of the fundamental groupoid and locally constant sheaves on $X$. It was observed by MacPherson that for stratified spaces, one can define an exit path category which in the same way classifies constructible sheaves ($1$-categorically). Treumann gave a $2$-categorical version of this result (\cite{Treumann}), and Lurie developed the $\infty$-categorical setting, defining the exit path $\infty$-category and generalising MacPherson's observation (\cite[Theorem A.9.3]{LurieHA}).

\medskip

\textbf{Main results.}
Let $\mathbf{G}$, $\Gamma$ and $X$ be as above and assume that $\mathbf{G}$ has positive $\Q$-rank so that $\Gamma\backslash X$ is non-compact. Let $\mathscr{P}$ denote the poset of rational parabolic subgroups of $\mathbf{G}$. For all $\mathbf{P}\in \mathscr{P}$, let $\mathbf{N_P}\leq \mathbf{P}$ denote the unipotent radical of $\mathbf{P}$ and write $\Gamma_{\mathbf{N_P}}:=\Gamma\cap \mathbf{N_P}(\Q)$.

\medskip

Our main theorem is the following.

\begin{xreftheorem}{exit path category of the reductive borel-serre compactification}
The exit path $\infty$-category of the reductive Borel--Serre compactification $\Gamma\backslash \XRBS$ is equivalent to the nerve of its homotopy category. This in turn is equivalent to the category $\mathscr{C}^{RBS}_\Gamma$ with objects the rational parabolic subgroups of $\mathbf{G}$ and hom-sets
\begin{align*}
\mathscr{C}^{RBS}_\Gamma(\mathbf{P},\mathbf{Q})=\{\gamma\in \Gamma\mid \gamma\mathbf{P}\gamma^{-1}\leq \mathbf{Q}\}/\Gamma_{\mathbf{N_P}},\quad\text{for all}\quad\mathbf{P},\mathbf{Q}\in \mathscr{P},
\end{align*}
where $\Gamma_{\mathbf{N_P}}$ acts by right multiplication, and composition is given by multiplication of representatives.
\end{xreftheorem}

The two important things to note here, is that the exit path $\infty$-category is equivalent to a $1$-category, and that the definition of this $1$-category makes no reference to the space $\Gamma\backslash \XRBS$, but is defined purely in terms of the poset of rational parabolic subgroups, their unipotent radicals and the conjugation action of $\Gamma$ on this poset. As an intermediate step towards this identification, we identify the exit path $\infty$-categories of the partial Borel--Serre compactification $\XBS$ of $X$ and the Borel--Serre compactification $\Gamma \backslash \XBS$ of $\Gamma\backslash X$. We also show that the equivalences can be chosen to be compatible with the quotient maps $\XBS\rightarrow \Gamma\backslash \XBS\rightarrow \Gamma\backslash \XRBS$ by choosing compatible basepoints.

\medskip

We have the following corollaries of the main theorem.

\begin{xrefcorollary}{weak homotopy equivalence rbs with crbs}
The reductive Borel--Serre compactification $\Gamma\backslash \XRBS$ is weakly homotopy equivalent to the geometric realisation of $\mathscr{C}^{RBS}_\Gamma$.
\end{xrefcorollary}

This immediately recovers the following result of Ji--Murty--Saper--Scherk (\cite[Corollary 5.3]{JiMurtySaperScherk}).

\begin{xrefcorollary}{fundamental group of rbs}
The fundamental group of the reductive Borel--Serre compactification $\Gamma\backslash \XRBS$ is isomorphic to the group $\Gamma/E_\Gamma$, where $E_\Gamma\triangleleft \Gamma$ is the normal subgroup generated by the subgroups $\Gamma_{\mathbf{N_P}}\leq \Gamma$ as $\mathbf{P}$ runs through all rational parabolic subgroups of $\mathbf{G}$.
\end{xrefcorollary}

For a stratified space $X$ and an associative ring $R$, let $D(\operatorname{Shv}_1(X, R))$ denote the classical derived category of sheaves on $X$ with values in left $R$-modules, and let $\operatorname{LMod}_R^1$ denote the category of left $R$-modules. Let $D_{\operatorname{cbl}}(\operatorname{Shv}_1(X, R))$ denote the full subcategory spanned by the complexes whose homology is constructible, and let $D_{\operatorname{cbl,cpt}}(\operatorname{Shv}_1(X, R))$ denote the full subcategory spanned by the complexes whose homology is constructible \textit{and} whose stalk complexes are perfect chain complexes. As another corollary, we get the following expression of these derived categories of sheaves as derived functor categories. 

\begin{xrefcorollary}{constructible derived category of RBS}
Let $R$ be an associative ring. There is an equivalence of categories
\begin{align*}
D_{\operatorname{cbl}}(\operatorname{Shv}_1(\Gamma\backslash \XRBS,R))\simeq D(\operatorname{Fun}(\mathscr{C}^{RBS}_{\Gamma},\operatorname{LMod}_R^1))
\end{align*}
which restricts to an equivalence
\begin{align*}
D_{\operatorname{cbl,cpt}}(\operatorname{Shv}_1(\Gamma\backslash \XRBS,R))\simeq D_{\operatorname{cpt}}(\operatorname{Fun}(\mathscr{C}^{RBS}_{\Gamma},\operatorname{LMod}_R^1)),
\end{align*}
where $D_{\operatorname{cpt}}(\operatorname{Fun}(\mathscr{C}^{RBS}_{\Gamma},\operatorname{LMod}_R^1))\subset D(\operatorname{Fun}(\mathscr{C}^{RBS}_{\Gamma},\operatorname{LMod}_R^1))$ is the full subcategory spanned by the complexes of functors $F_\bullet$ such that $F_\bullet(x)$ is a perfect complex for all $x\in X$.
\end{xrefcorollary}

This can be interpreted as a combinatorial incarnation of constructible complexes of sheaves on the reductive Borel--Serre compactification. In fact, we get an $\infty$-categorical result (see \Cref{constructible derived category of RBS infinity}), but we state the $1$-categorical consequence here as this speaks of the more classical constructible derived category which has been studied in for example \cite{Saper05a,Saper05b} and \cite{GoreskyHarderMacPherson}.

\medskip

The definition of the category $\mathscr{C}^{RBS}_\Gamma$ generalises to a purely algebraic setting of a group acting on a poset. We make this precise in \Cref{Groups acting on posets} and note that this generalisation recovers some well-known categories appearing in the literature: in the setting of finite groups with a split BN-pair of characteristic $p$, we recover (the opposite of) the orbit category on the collection of $p$-radical subgroups, an object that has been studied extensively in finite group theory (\cite{Alperin, AlperinFong, Bouc, JackowskiMcClureOliverI, JackowskiMcClureOliverII, Grodal02, Grodal18}, see also \Cref{orbit categories and p-radical subgroups}); for the mapping class group acting on (the opposite of) the augmented curve complex, we recover (the opposite of) the Charney--Lee category introduced in \cite{CharneyLee84} and appearing also in \cite{EbertGiansiracusa} and \cite{ChenLooijenga} where it is shown that the homotopy type of the Deligne--Mumford compactification is given by this category. We believe that the Charney--Lee category is in fact (the opposite of) the exit path category of the Deligne-Mumford compactification which would strengthen the results of \cite{CharneyLee84,EbertGiansiracusa,ChenLooijenga}. These types of categories also appear in joint work with Dustin Clausen, where they are introduced as models for unstable algebraic K-theory (see further comments below).

\medskip

\textbf{Calculational tools.}
In order to identify the exit path $\infty$-category of the reductive Borel--Serre compactification, we develop some calculation tools. Following ideas of Woolf (\cite{Woolf}), we identify the mapping spaces in the exit path $\infty$-category as the fibres of certain fibrations, namely the end point evaluation fibrations out of the so-called \textit{homotopy link} from the theory of homotopically stratified sets developed by Quinn (\cite{Quinn88}). With a little extra data, the resulting long exact sequences of homotopy groups enable us to identify the homotopy types of the mapping spaces. In particular, we can use this to determine whether the mapping spaces have contractible components, implying that the exit path $\infty$-category is equivalent to the nerve of its homotopy category.

\medskip

We go on to study group actions on stratified spaces and to identify the exit path $\infty$-categories of the resulting quotient stratified spaces in particularly nice cases. This is done in Theorems \ref{exit path categories and group actions} and \ref{exit path categories and group actions II} and we believe that these results should be applicable to a larger class of interesting stratified spaces. They should moreover be compared with a result of Chen--Looijenga (\cite[Theorem 1.7]{ChenLooijenga}): we rephrase and slightly strengthen their result in certain situations, and the conditions that need to be satisfied for our result to apply are a local version of their conditions (see also \Cref{compare Chen-Looijenga}). It should be stressed, however, that the settings differ a great deal and that, apart from allowing local conditions and local data, ours is the more restrictive setting since we do not deal with the cases where the quotient stratified space is an orbispace.

\medskip

\textbf{Further work.}
In joint work with Dustin Clausen, we investigate how generalisations of the categories $\mathscr{C}^{RBS}_\Gamma$ model unstable algebraic K-theory (\cite{ClausenOrsnesJansen}). For an associative ring $A$ and any finitely generated projective $A$-module $M$, we introduce a category $\operatorname{RBS}(M)$ which is defined purely in terms of linear algebra internal to $M$ and which naturally generalises $\mathscr{C}^{RBS}_\Gamma$. We show that if $A$ is a ring with many units, e.g. a local ring with infinite residue field, and every split submodule of M is free, then the geometric realisation $|\operatorname{RBS}(M)|$ recovers the plus-construction of $B\!\operatorname{GL}(M)$. In the case of finite fields, the model we introduce is in a certain sense better than the one given by the plus-construction as it gets rid of the complicated unstable $\F_p$-homology of $\operatorname{GL}_n(k)$ for $k$ a finite field of characteristic $p$. We also show that these categories in a natural way stabilise to provide a model for the (stable) algebraic K-theory space $K(A)$.

\medskip

\textbf{Acknowledgements.}
I would like to thank Dustin Clausen for proposing the question that has resulted in this paper and for the many fruitful conversations we have had related to this work. This paper is part of my PhD thesis at the University of Copenhagen and I would also like to thank my advisor Søren Galatius for his encouragement and helpful discussions throughout my PhD studies.

\medskip

\textbf{Notation and conventions.}
By \textit{$\infty$-category} we mean quasicategory, that is, a simplicial set satisfying the extension property for all inner horn inclusions. We refer to \cite{LurieHTT} for details.

\medskip

By the \textit{homotopy category} of an $\infty$-category $\mathscr{C}$, we mean the $1$-category $h\mathscr{C}$ with objects the $0$-simplices of $\mathscr{C}$ and morphisms the $1$-simplices subject to certain relations given by the $2$-simplices (see \cite[\S 1.2.3]{LurieHTT}).

\medskip

Given a $(1)$-category $\mathscr{C}$, we define its \textit{geometric realisation} $|\mathscr{C}|$ as the geometric realisation of its nerve $N_\bullet\mathscr{C}$. We reserve the term \textit{classifying space} for groups and monoids, i.e.~the classifying space $BM$ of a group $M$ is the geometric realisation of the one object category with morphisms the elements of $M$.

\medskip

For a group $G$, an element $g\in G$ and a subgroup $H\leq G$, we write ${}^gH=gHg^{-1}$ and $H^g=g^{-1}Hg$ for the conjugated subgroups, and similarly for algebraic groups and subgroups.

\medskip

For a set $A$ and a subset $B\subset A$, we denote the set difference by $A-B$.

\medskip

We write $[n]=\{0<1<\cdots<n\}$ for the linearly ordered poset with $n+1$ elements.

\section{Stratified homotopy theory}\label{exitcat}

We recall the definitions of conically stratified (poset-stratified) spaces and exit path $\infty$-categories following Lurie (\cite[Appendix A]{LurieHA}). We also recall the homotopy link from the theory of homotopically stratified sets introduced by Quinn (\cite{Quinn88}) and two important properties of this space, as these will be essential for developing our calculational tools. We state the main property of the exit path $\infty$-category, namely the fact that constructible sheaves are equivalent to representations of it, and we show that this result can be extended to the constructible derived category.

\subsection{Poset-stratified spaces} In the following, all posets are topologised by the Alexandroff topology, i.e.~the open sets are the upwards closed sets. For a given poset $I$ and a fixed $i\in I$, we will write $I_{>i}=\{j\in I\mid j>i\}$, and similarly $I_{\geq i}$, $I_{<i}$, $I_{\leq i}$.

\begin{definition}
A \textit{poset-stratified space} (or simply \textit{stratified space}) is a continuous map $s\colon X\rightarrow I$, where $X$ is a topological space and $I$ a poset. The poset $I$ is called the \textit{poset of strata} and the subspace $X_i=s^{-1}(i)$ is called the $i$'th \textit{stratum}. A \textit{stratum preserving} (or \textit{stratified}) map from $s\colon X\rightarrow I$ to $r\colon Y\rightarrow J$ is a pair of continuous maps $(f\colon X\rightarrow Y,\theta\colon I\rightarrow J)$ such that $r\circ f= \theta\circ s$.
\defend
\end{definition}

\begin{remark}
When no confusion can occur, we omit the poset of strata and refer to a stratified space $s\colon X\rightarrow I$ simply by $X$. If we want to stress that $X$ is stratified over the poset $I$, then we say that $X$ is an $I$-stratified space. Similarly, when considering a stratum preserving map $(f,\theta)$, we may omit the order-preserving map $\theta$.

\medskip

The strata $X_i$, $i\in I$, define a partition of $X$, and continuity of $s$ is equivalent to requiring the upward unions of strata to be open in $X$, i.e.~$\bigcup_{j\geq i}X_j\subset X$ is open for all $i\in I$. The closure relations in the partition translate to poset relations in $I$: if $X_i\subseteq \overline{X_j}$, then $i\leq j$. It will often be the case in naturally occurring examples that $X_i\subseteq \overline{X_j}$ if and only if $i\leq j$.
\exend
\end{remark}

\begin{definition}\label{stratified cone}
Let $s\colon Y\rightarrow I$ be a stratified space. The \textit{(open) cone} on $Y$ is the stratified space $s^{\triangleleft}\colon C(Y)\rightarrow I^{\triangleleft}$ defined as follows: the poset of strata is $I^{\triangleleft}:=I\cup \{-\infty\}$ with $-\infty\leq i$ for all $i\in I$; as a set $C(Y)=(Y\times (0,1))\coprod *$, and $U\subseteq C(Y)$ is open, if and only if:
\begin{enumerate}[label=(\roman*)]
\item $U\cap (Y\times (0,1))$ is open;
\item $\ast\in U$ implies $Y\times (0,\epsilon)\subseteq U$ for some $\epsilon>0$.
\end{enumerate}
The stratification map $s^{\triangleleft}$ is given by $s^\triangleleft(x,t)=s(x)$ and $s^\triangleleft(\ast)=-\infty$.
\defend
\end{definition}

\begin{remark}
The topology of $C(Y)$ above coincides with the teardrop topology on the open cone of $Y$ (see for example \cite{HughesTaylorWeinbergerWilliams}). If $Y$ is compact Hausdorff, then $C(Y)$ is homeomorphic to the pushout $(Y\times [0,1))\coprod_{Y\times \{0\}}\ast$.

\medskip

Note that if $Y$ is metrisable, then by \cite[Lemma 3.15]{HughesTaylorWeinbergerWilliams}, so is the stratified cone $C(Y)$.
\exend
\end{remark}

\begin{definition}
An $I$-stratified space $X$ is \textit{conically stratified at} $x\in X_i\subseteq X$, if there exists:
\begin{enumerate}[label=(\roman*)]
\item a topological space $V$,
\item an $I_{>i}$-stratified space $L$,
\item and a stratified homeomorphism $(\phi,\theta)\colon V\times C(L)\xrightarrow{\ \cong \ } U$ onto a neighbourhood $U$ of $x$ in $X$, where $\theta\colon (I_{>i})^{\triangleleft}\rightarrow I$ is the canonical identification of $(I_{>i})^{\triangleleft}$ with $I_{\geq i}\subseteq I$.
\end{enumerate}
We say that $X$ is \textit{metrisably conically stratified} at $x$, if there exists a conical neighbourhood as above such that the union $V\cup X_j$ is metrisable for all $j>i$.

\medskip

A stratified space $X$ is \textit{conically stratified} if it is conically stratified at all points, and it is \textit{metrisably conically stratified} if it is metrisably conically stratified at all points.
\defend
\end{definition}

\begin{remark}
We will often write that a point $x\in X_i$ admits a conical neighbourhood $\phi\colon V\times C(L)\xrightarrow{\cong} U$, in which case we implicitly assume that $L$ is an $I_{>i}$-stratified space, $V=U\cap X_i$ and $\phi$ is a stratified homeomorphism identifying $I_{>i}^{\triangleleft}$ with $I_{\geq i}$.

\medskip

We call $L$ a \textit{link space} of $x$ in $X$. We cannot in general speak of \textit{the} link space of $x$, although in many cases it will be well-defined up to some sort of equivalence. If a stratified space $X$ is equipped with link bundles (in the sense of \cite{GoreskyHarderMacPherson}), then $X$ is conically stratified and link spaces of any two points $x,x'\in X_i$ are stratified homeomorphic. If $X$ is homotopically stratified (in the sense of \cite{Quinn88}), then $X$ is conically stratified and link spaces of any two points $x,x'\in X_i$ are homotopy equivalent.
\exend
\end{remark}

\begin{example}
\label{manifolds with corners}
Let $M$ be a smooth manifold with corners of dimension $n$, i.e.~a space modelled smoothly upon open subsets of a quadrant in $\R^n$. A point $x\in M$ has \textit{index} $j$, if there is a chart $(U,\phi)$ on $M$, such that $\phi(x)$ has exactly $j$ coordinates equal to zero. Let $M_j\subseteq M$ denote the subspace consisting of points of index $j$; it is a smooth manifold of dimension $n-j$. The standard stratification of $M$ as a manifold with corners is by the path components of the $M_j$, $j=0,\ldots, n$. Let $N\subseteq M_j\subseteq M$ be a stratum, that is a path component, and let $x\in N$. $N$ is of codimension $j$ in $M$, and there is a conical neighbourhood $x\in U\cong V\times C(\Delta^{j-1})$, where $\Delta^{j-1}$ is the standard $(j-1)$-simplex stratified as a manifold with corners.

\medskip

Manifolds with corners have more rigourous structure, namely mapping cylinder neighbourhoods of each stratum, not just conical neighbourhoods of points --- this is the case for many naturally occurring stratified spaces, but we will only need to local data, so we refrain from going into this.
\exend
\end{example}

\begin{definition}\label{standard stratified n-simplex}
The \textit{standard stratified $n$-simplex} is the standard $n$-simplex
\begin{align*}
\Delta^n=\{(t_0,\ldots,t_n)\in [0,1]^n\mid \textstyle\sum t_i=1\}
\end{align*}
stratified by the map $s_n\colon \Delta^n\rightarrow [n]$ defined by
\begin{align*}
s_n(t_0,\ldots,t_k,0,\ldots,0)=k \quad \text{if } t_k\neq 0.
\end{align*}
In other words, $s_n$ maps the subspace $\Delta^{0,\ldots,k}-\Delta^{0,\ldots,k-1}\subset \Delta^n$ to $k$, where $\Delta^{0,\ldots,j}$ denotes the face spanned by the vertices $0,1,\ldots,j$.
\defend
\end{definition}

\begin{remark}
Note that the standard stratified simplex is not stratified as a manifold with corners, but rather in a way that retains the combinatorial information. It can be identified with the $(n+1)$-fold stratified closed mapping cone of a point, where the closed mapping cone is the stratified space obtained by replacing $(0,1)$ by $(0,1]$ in \Cref{stratified cone}.
\exend
\end{remark}

\subsection{Homotopy links}

It will be convenient for us to use a homotopical version of link spaces, namely the homotopy link defined by Quinn in his study of homotopically stratified sets (\cite{Quinn88}).

\begin{definition}
Let $X$ be a topological space and $Y\subseteq X$ a subspace. The \textit{homotopy link} of the pair $(X,Y)$ is the subspace
\begin{align*}
H(X,Y)=\{\gamma\colon [0,1]\rightarrow X\mid \gamma(0)\in Y,\ \gamma((0,1])\subseteq X-Y\}\subseteq C([0,1],X)
\end{align*}
of the path space of $X$ equipped with the compact-open topology.
\defend
\end{definition}

Let $X$ be a topological space and $Y\subseteq X$ a closed subspace. If we stratify $X$ over $\{0<1\}$ by sending $Y$ to $0$ and $X-Y$ to $1$, then the points of $H(X,Y)$ can be identified with the stratum preserving maps $\sigma \colon \Delta^1\rightarrow X$ starting in $Y$ and ending in $X-Y$.

\medskip

We need two fundamental facts about the homotopy link which hold when the pair $(X,Y)$ is sufficiently nice:
\begin{enumerate}[label=(\roman*)]
\item the end point evaluation map
\begin{align*}
e\colon H(X,Y)\rightarrow Y\times (X-Y),\quad \gamma\mapsto (\gamma(0), \gamma(1)),
\end{align*}
is a fibration (see \Cref{e fibration top pair} and \Cref{e fibration conically stratified space});
\item the homotopy link serves as a homotopical replacement for link spaces or more generally for neighbourhoods admitting a nearly strict deformation retraction (see \Cref{holink homotopy equivalent to neighbourhood} and \Cref{link eq to htpy link}).
\end{enumerate}
These results are well-known, but we have been unable to locate a source which does not work in a much more general or slightly different setting so for the sake of self-containment, we have chosen to include the proofs in this fairly elementary point-set topological setting (see \Cref{appendixA}). These proofs also explain our need to work with \textit{metrisably} conically stratified spaces.

\medskip

The following is a direct consequence of \Cref{e fibration top pair}, since the evaluation at zero map $e_0\colon H(Y\times C(Z),Y\times \{*\})\rightarrow Y$ is a fibration for any topological spaces $Y$ and $Z$.

\begin{proposition}\label{e fibration conically stratified space}
Let $X$ be an $I$-stratified space. Let $i\in I$, $x\in X_i$ and $j>i$. Suppose $x$ has a conical neighbourhood $U$ and set $V=U\cap X_i$. If the union $V\cup X_j$ is metrisable, then the end point evaluation map $e\colon H(X_j\cup V,V)\rightarrow V\times X_j$ is a fibration.
\end{proposition}

Let $X$ be an $I$-stratified space and suppose $x\in X_i$ has a conical neighbourhood
\begin{align*}
\phi\colon V\times C(L)\xrightarrow{\cong} U.
\end{align*}
Then the map
\begin{align*}
U\times [0,1]\rightarrow U, \quad (\phi(v,[l,s]),t)\mapsto \phi(v,[l,st])
\end{align*}
is a (stratum preserving) nearly strict deformation retraction into $V=U\cap X_i$ (\Cref{nearly strict deformation retraction}). For all $i<j$ and any fixed $\epsilon\in (0,1)$, we have a map 
\begin{align*}
\Psi_{\phi,\epsilon}^j\colon V\times L_j &\rightarrow H(X_j\cup V,V),\quad (v,l)\mapsto \gamma_{v,l,\epsilon}, \\
&\text{where}\quad \gamma_{v,l,\epsilon}\colon [0,1]\rightarrow X_j\cup V,\quad t\mapsto \phi(v,[l,t\epsilon]).
\end{align*}
That is, $\gamma_{v,l,\epsilon}$ is the path tracing the cone coordinate from the apex to $\epsilon$ for fixed $v\in V$ and $l\in L$ in the other coordinates. The following is a direct application of \Cref{holink homotopy equivalent to neighbourhood}, since the map $\Psi_{\phi,\epsilon}^j$ is the composition of the following three maps:
\begin{enumerate}[label=(\roman*)]
\item the inclusion at $\epsilon$, $V\times L_j\rightarrow V\times L_j\times (0,1)$, $(v,l)\mapsto (v,l,\epsilon)$;
\item the homeomorphism $V\times L_j\times (0,1)\xrightarrow{\ \cong\ } U\cap X_j$ given by $\phi$;
\item and the map $U\cap X_j\rightarrow H(X_j\cup V,V)$ of \Cref{holink homotopy equivalent to neighbourhood}.
\end{enumerate}

\begin{proposition}\label{link eq to htpy link}
Let $X$ be an $I$-stratified space. Let $i\in I$, $x\in X_i$ and $j>i$. Suppose $x$ has a conical neighbourhood $\phi\colon V\times C(L)\xrightarrow{\cong} U$.  If the union $V\cup X_j$ is metrisable, then for any choice of $\epsilon\in (0,1)$, the map $\Psi_{\phi,\epsilon}^j\colon V\times L_j \rightarrow H(X_j\cup V,V)$ defined above is a homotopy equivalence. In particular, if $V$ is weakly contractible, then the map $L_j\rightarrow H(X_j\cup V,V)$, $l\mapsto \gamma_{x,l,\epsilon}$, is a weak homotopy equivalence.
\end{proposition}

\subsection{Exit path \texorpdfstring{$\infty$}{infinity}-categories and constructible sheaves}

Following the work of Lurie, we recall the definitions of the exit path $\infty$-category and the classification of constructible sheaves as representations of the exit path $\infty$-category (\cite[Appendix A]{LurieHA}).

\medskip

In the following $s_n\colon \Delta^n\rightarrow [n]$ will denote the standard stratified $n$-simplex as defined in \Cref{standard stratified n-simplex}, and $\operatorname{Sing}(X)$ denotes the singular set of a topological space.

\begin{definition}
Let $s\colon X\rightarrow I$ be a conically stratified space. The \textit{exit path $\infty$-category} of $X\rightarrow I$ is the subsimplicial set $\exit_\infty(X,I)\subset \operatorname{Sing}(X)$ whose $n$-simplices are the maps $\sigma\colon \Delta^n\rightarrow X$ for which there is an order preserving map $\theta\colon [n]\rightarrow I$ such that $s\circ \sigma=\theta\circ s_n$.
\end{definition}

\begin{remark}
If $\theta$ and $\theta'$ satisfy $\theta\circ s_n=\theta'\circ s_n$, then $\theta=\theta'$, so we can also define the exit path $\infty$-category as the simplicial set with $n$-simplices the stratum preserving maps $\sigma\colon \Delta^n\rightarrow X$.

\medskip

The stratified spaces considered in this paper come equipped with natural stratifications, so from now on, we write $\exit_\infty(X)=\exit_\infty(X,I)$, letting the poset $I$ be implicit in the notation. 
\exend
\end{remark}

The following theorem justifies the name and notation.

\begin{theorem}[{\cite[Theorem A.6.4]{LurieHA}}]
For a conically stratified space $X$, the simplicial set $\exit_\infty(X)$ is an $\infty$-category.
\end{theorem}

Thus we have a functor $\exit_\infty\colon \operatorname{Strat} \rightarrow \operatorname{Cat}_\infty$ (of $1$-categories).

\begin{definition}
We define the \textit{exit path $1$-category} of a conically stratified space $X$ as the homotopy category of the exit path $\infty$-category of $X$ and we denote it by $\exit_1(X)$.
\defend
\end{definition}

\begin{remark}
Note that we are not taking the enriched homotopy category, but just the underlying $1$-category; we deal with the higher homotopy in the mapping spaces of $\exit_\infty(X)$ separately.
\exend
\end{remark}

The following remark should provide some intuition for the exit path $\infty$-category.

\begin{remark}
Let $X$ be a conically $I$-stratified space.
\begin{itemize}[label=$\ast$]
\item The $0$-simplices of $\exit_\infty(X)$ are the points of $X$.
\item Identifying $\Delta^1\cong [0,1]$, the $1$-simplices of $\exit_\infty(X)$ are the paths $\sigma\colon [0,1]\rightarrow X$ which satisfy $\sigma(0)\in X_i$ and $\sigma((0,1])\subseteq X_j$ for some $i\leq j$ in $I$. In other words, the exit paths either stay within one stratum or leave the deeper stratum instantaneously entering the stratum containing the end point. We see that the homotopy link $H(X_i\cup X_j,X_i)$ of $X_i$ in $X_i\cup X_j$ is a subset of the $1$-simplices of $\exit_\infty(X)$.
\item The morphisms in $\exit_1(X)$ are represented by $1$-simplices of $\exit_\infty(X)$ as described above, but composition is hard to describe concretely. Intuitively, however, we can think of the composite of two such paths in $\exit_1(X)$ as the concatenation.
\item For all $i\in I$, $\exit_\infty(X)$ contains the fundamental $\infty$-groupoid $\operatorname{Sing}(X_i)$ as the full subcategory spanned by the points of $X_i$. \exend
\end{itemize}
\end{remark}

The most important feature of the exit path $\infty$-category is that for sufficiently well-behaved stratified spaces, it classifies constructible sheaves. We state this for sheaves with values in any compactly generated $\infty$-category. Lurie proves it for space-valued sheaves, but the generalisation is well-known and quite elementary to prove. However, since we have been unable to locate a proof in the literature, we have included a detailed proof in the appendix, also in the hope that it makes these results more accessible to a reader without a background in $\infty$-categories. We refer to \Cref{appendix: sheaves} and \cite[Section A.5]{LurieHA} for proofs and details. 

\medskip

For a topological space $X$ and a compactly generated $\infty$-category $\mathscr{C}$, we denote by $\operatorname{Shv}(X,\mathscr{C})$ the $\infty$-category of $\mathscr{C}$-valued sheaves on $X$ (see \Cref{C-valued sheaves}).

\begin{definition}
Let $X$ be an $I$-stratified space and let $\mathscr{C}$ be a compactly generated $\infty$-category. A sheaf $\mathcal{F}\in \operatorname{Shv}(X,\mathscr{C})$ is \textit{constructible} if for every $i\in I$, the restriction $\mathcal{F}\vert_{X_i}$ is a locally constant sheaf in $\operatorname{Shv}(X_i,\mathscr{C})$. We denote by $\operatorname{Shv}_{\operatorname{cbl}}(X,\mathscr{C})$ the full subcategory spanned by the constructible sheaves.
\defend
\end{definition}

We will need impose some condition on the stratifying poset in order to classify constructible sheaves in terms of the exit path $\infty$-category:

\begin{definition}
A poset $I$ is said to satisfy the \textit{ascending chain condition} if every non-empty subset of $I$ has a maximal element.
\defend
\end{definition}

The following theorem generalises the monodromy equivalence which classifies locally constant sheaves as representations of the fundamental $\infty$-groupoid.

\begin{theorem}[{\cite[Theorem A.9.3]{LurieHA} and \Cref{exit paths and constructible sheaves equivalence arbitrary coefficients}}]\label{exit paths classify constructible sheaves}
Let $\mathscr{C}$ be a compactly generated $\infty$-category. Suppose $X$ is a conically $I$-stratified space which is paracompact and locally contractible, and that $I$ satisfies the ascending chain condition. Then there is an equivalence of $\infty$-categories
\begin{align*}
\Psi_X\colon \operatorname{Fun}(\exit_\infty(X), \mathscr{C})\rightarrow \operatorname{Shv}_{\operatorname{cbl}}(X,\mathscr{C}).
\end{align*}
\end{theorem}

\begin{remark}
The result is stated in \cite{LurieHA} for spaces which are locally of singular shape, but we wish to avoid going into the technicalities involved in defining this notion here, so we restrict ourselves to locally contractible spaces.
\exend
\end{remark}

We have the following corollary.

\begin{corollary}[{\cite[Corollary A.9.4]{LurieHA}}]\label{exit and sing weak htpy equiv}
Suppose $X$ is a conically $I$-stratified space which is paracompact and locally contractible and where $I$ satisfies the ascending chain condition. The inclusion $\exit_\infty(X)\hookrightarrow \operatorname{Sing}(X)$ is a weak homotopy equivalence of simplicial sets, i.e.~the induced map of geometric realisations $|\exit_\infty(X)|\rightarrow |\operatorname{Sing}(X)|$ is a homotopy equivalence.
\end{corollary}

\subsection{The constructible derived category of sheaves}\label{constructible derived category}

If the exit path $\infty$-category is equivalent to the nerve of its homotopy category, then the classification of constructible sheaves as representations of the exit path $\infty$-category can be extended to give an expression of the constructible derived category of sheaves (of $R$-modules) in terms of the exit path $1$-category. The observations made in this section are quite elementary for anyone with a background in $\infty$-categories. We have chosen to be quite detailed for the sake of other potential readers.

\medskip

For a Grothendieck abelian category $\mathcal{A}$ we denote by $\mathscr{D}(\mathcal{A})$ the (unbounded) derived $\infty$-category of $\mathcal{A}$ (see \cite[\S 1.3]{LurieHA}). The homotopy category of $\mathscr{D}(\mathcal{A})$ is the classical (unbounded) derived ($1$-)category $D(\mathcal{A})$ of $\mathcal{A}$.

\medskip

Let $R$ be an associative ring and let $\operatorname{LMod}_R^1$ denote the $1$-category of left $R$-modules. Viewing $R$ as a discrete $\mathbb{E}_1$-ring, let $\operatorname{LMod}_R$ denote the $\infty$-category of left $R$-module spectra. Then
\begin{align*}
\mathscr{D}(\operatorname{LMod}^1_R)\xrightarrow{\simeq} \operatorname{LMod}_R
\end{align*}
by \cite[Proposition 7.1.1.16]{LurieHA}. In particular, the derived category $D(R):=D(\operatorname{LMod}_R^1)$ is equivalent to the homotopy category of $\operatorname{LMod}_R$. By \cite[Proposition 7.2.4.2]{LurieHA}, the $\infty$-category $\operatorname{LMod}_R$ is compactly generated and the subcategory of compact objects is the $\infty$-category $\operatorname{Perf}_\infty(R)$ of perfect modules (\cite[\S 7.2.4]{LurieHA}). Under the equivalence above, $\mathscr{D}(\operatorname{LMod}^1_R)\xrightarrow{\simeq} \operatorname{LMod}_R$, perfect modules correspond to perfect chain complexes, i.e. complexes which are quasi-isomorphic to bounded chain complexes whose terms are finitely generated projective modules (Corollary 7.2.4.5 and Example 7.2.4.25 of \cite{LurieHA}). Let $\operatorname{Perf}_1(R)\subseteq D(R)$ denote the full subcategory spanned by the perfect chain complexes.

\medskip

Let $\operatorname{Shv}_1(X,R)$ denote the $1$-category of sheaves on $X$ with values in $\operatorname{LMod}_R^1$. This is a Grothendieck abelian category, and we consider the derived $\infty$-category $\mathscr{D}(\operatorname{Shv}_1(X,R))$.

\begin{remark}
The canonical functor
\begin{align*}
\mathscr{D}(\operatorname{Shv}_1(X,R))\rightarrow \operatorname{Shv}(X,\mathscr{D}(R))\simeq\operatorname{Shv}(X,\operatorname{LMod}_R)
\end{align*}
is fully faithful with essential image the full subcategory $\operatorname{Shv}^{\operatorname{hyp}}(X,\operatorname{LMod}_R)$ of hypercomplete sheaves, that is, sheaves which satisfy descent with respect to any hypercovering not just covering sieves (\cite[\S 6.5.2]{LurieHTT}, see also the discussion at \cite{mathoverflow265557}). Constructible sheaves are hypercomplete by \cite[Proposition A.5.9]{LurieHA}, and we note that the subcategory 
\begin{align*}
\operatorname{Shv}_{\operatorname{cbl}}(X,\operatorname{LMod}_R)\subseteq \operatorname{Shv}^{\operatorname{hyp}}(X,\operatorname{LMod}_R)
\end{align*}
corresponds to the full subcategory $\mathscr{D}_{\operatorname{cbl}}(\operatorname{Shv}_1(X,R))\subseteq\mathscr{D}(\operatorname{Shv}_1(X,R))$ spanned by the complexes whose homology sheaves are constructible. Similarly, we see that the subcategory of constructible compact-valued sheaves (i.e.~whose stalk complexes are compact objects in $\operatorname{LMod}_R$, see \Cref{constructible compact-valued sheaf})
\begin{align*}
\operatorname{Shv}_{\operatorname{cbl,cpt}}(X,\operatorname{LMod}_R)\subseteq \operatorname{Shv}_{\operatorname{cbl}}(X,\operatorname{LMod}_R)
\end{align*}
corresponds to the subcategory $\mathscr{D}_{\operatorname{cbl,cpt}}(\operatorname{Shv}_1(X,R))\subseteq\mathscr{D}_{\operatorname{cbl}}(\operatorname{Shv}_1(X,R))$ spanned by the complexes whose homology sheaves are constructible \textit{and} whose stalk complex is a perfect chain complex.\exend
\end{remark}

\begin{definition}
The \textit{constructible derived category of sheaves} on $X$ with values in left $R$-modules is the full subcategory
\begin{align*}
D_{\operatorname{cbl}}(\operatorname{Shv}_1(X,R))\subseteq D(\operatorname{Shv}_1(X,R))
\end{align*}
spanned by the complexes of sheaves with constructible homology sheaves. The \textit{constructible compact-valued derived category of sheaves} on $X$ with values in left $R$-modules is the full subcategory
\begin{align*}
D_{\operatorname{cbl,cpt}}(\operatorname{Shv}_1(X,R))\subseteq D(\operatorname{Shv}_1(X,R))
\end{align*}
spanned by the complexes of sheaves with constructible homology sheaves and whose stalk complex is a perfect chain complex.
\defend
\end{definition}

We give two examples which are of interest in the study of the reductive Borel-Serre compactification, but first we make the following observation.

\begin{remark}
Suppose $R$ is a regular Noetherian ring of finite Krull dimension. Then it has finite global dimension, and thus any bounded below chain complex whose terms are finitely generated is quasi-isomorphic to a bounded complex whose terms are finitely generated projective. Therefore a constructible complex of sheaves in the sense of \cite[\S 1.4]{GoreskyMacPherson83} is a constructible compact-valued sheaf in the sense of \Cref{constructible compact-valued sheaf}.
\exend
\end{remark}

\begin{example}\label{examples of constructible complexes of sheaves}
\ 
\begin{enumerate}[label=(\roman*)]
\item Let $X$ be a topological pseudomanifold with a fixed stratification and let $k$ be a field. Intersection homology of $X$ can be defined as the hypercohomology of a complex of sheaves $\mathbf{IC}_p(X)$ on $X$ taking values in $k$-vector spaces (\cite{GoreskyMacPherson80}, \cite{GoreskyMacPherson83}). The complexes $\mathbf{IC}_p(X)$ are constructible and compact-valued \cite[\S 3]{GoreskyMacPherson83}.
\item The weighted cohomology of the arithmetic group $\Gamma$ is defined as the hypercohomology of a complex of sheaves $\mathbf{W}^p\mathbf{C}^\bullet(\mathbf{E})$ on the reductive Borel--Serre compactification associated with $\Gamma$ taking values in complex vector spaces (\cite{GoreskyHarderMacPherson}). The complexes $\mathbf{W}^p\mathbf{C}^\bullet(\mathbf{E})$ are constructible and compact-valued \cite[Theorem 17.6]{GoreskyHarderMacPherson}.\exend
\end{enumerate}
\end{example}

We have the following theorem.

\begin{theorem}\label{equivalence of derived infinity-categories}
Let $X$ be a paracompact, locally contractible conically $I$-stratified space with $I$ satisfying the ascending chain condition, and let $R$ be an associative ring. Suppose the exit path $\infty$-category $\exit_\infty(X)$ is equivalent to the nerve of its homotopy category $\exit_1(X)$. Then there is an equivalence of $\infty$-categories
\begin{align*}
\operatorname{Shv}_{\operatorname{cbl}}(X,\operatorname{LMod}_R)\simeq \mathscr{D}(\operatorname{Fun}\big(\exit_1(X),\operatorname{LMod}_R^1)),
\end{align*}
which restricts to an equivalence
\begin{align*}
\operatorname{Shv}_{\operatorname{cbl,cpt}}(X,\operatorname{LMod}_R)\simeq \mathscr{D}_{\operatorname{cpt}}(\operatorname{Fun}\big(\exit_1(X),\operatorname{LMod}_R^1)),
\end{align*}
where $\mathscr{D}_{\operatorname{cpt}}(\operatorname{Fun}\big(\exit_1(X),\operatorname{LMod}_R^1))$ is the full subcategory spanned by the complexes of functors $F_\bullet$ such that $F_\bullet(x)$ is a perfect complex for all $x\in X$.
\end{theorem}
\begin{proof}
Propositions 1.3.4.25 and 1.3.5.15 of \cite{LurieHA} give us the first of the following two equivalences, and the second is the one of \Cref{exit paths classify constructible sheaves}.
\begin{align*}
\mathscr{D}(\operatorname{Fun}(\exit_1(X),\operatorname{LMod}_R^1))\xrightarrow{\simeq} \operatorname{Fun}(\exit_\infty(X),\operatorname{LMod}_R)\xrightarrow{\simeq}  \operatorname{Shv}_{\operatorname{cbl}}(X,\operatorname{LMod}_R).
\end{align*}
The restriction to compact objects is a consequence of \Cref{exit paths and constructible sheaves equivalence arbitrary coefficients with compact stalks}.
\end{proof}

Taking homotopy categories, we get the following corollary.

\begin{corollary}\label{equivalence of derived 1-categories}
In the situation of \Cref{equivalence of derived infinity-categories} there is an equivalence of $1$-categories
\begin{align*}
D_{\operatorname{cbl}}(\operatorname{Shv}_1(X,R))\simeq D(\operatorname{Fun}(\exit_1(X),\operatorname{LMod}_R^1))
\end{align*}
which restricts to an equivalence
\begin{align*}
D_{\operatorname{cbl,cpt}}(\operatorname{Shv}_1(X,R))\simeq D_{\operatorname{cpt}}(\operatorname{Fun}(\exit_1(X),\operatorname{LMod}_R^1))
\end{align*}
where $D_{cpt}(\operatorname{Fun}(\exit_1(X),\operatorname{LMod}_R^1))$ is the full subcategory spanned by the complexes of functors $F_\bullet$ such that $F_\bullet(x)$ is a perfect complex for all $x\in X$.
\end{corollary}

\section{Calculational tools}\label{calculational tools}

If $X$ is a metrisably conically stratified space, then the end point evaluation maps from appropriately chosen homotopy links are fibrations (\Cref{e fibration conically stratified space}). We identify the mapping spaces of the exit path $\infty$-category $\exit_\infty(X)$ with the fibres of these fibrations and exploit the resulting long exact sequences of homotopy groups. This follows ideas of Woolf (\cite{Woolf}). We apply these tools to determine the exit path $\infty$-category of quotients of sufficiently contractible stratified spaces under well-behaved group actions --- this recovers and strengthens results of of Chen--Looijenga (\cite{ChenLooijenga}).

\subsection{Mapping spaces, fibrations and long exact sequences}

Recall the definition of the homotopy link given in \Cref{exitcat}: for a topological space $X$ and a subspace $Y\subseteq X$, the \textit{homotopy link} of $Y$ in $X$ is the subspace of paths (equipped with the compact-open topology)
\begin{align*}
H(X,Y)=\{\gamma\colon I\rightarrow X\mid \gamma(0)\in Y,\ \gamma((0,1])\subseteq X-Y\}\subset C(I,X).
\end{align*}

We have already observed that the homotopy link is a subset of the $1$-simplices in the exit path $\infty$-category. It turns out that the mapping spaces in the exit path $\infty$-category can be identified with subspaces of the homotopy link. Note that we are not yet requiring the stratified space to be \textit{metrisably} conically stratified.

\begin{proposition}\label{M cong F}
Let $X$ be a conically $I$-stratified space, let $i< j$ in $I$ and choose $x\in X_i$, $y\in X_j$. Let $V$ be a neighbourhood of $x$ in $X_i$. The mapping space $M(x,y)$ of the exit path $\infty$-category $\exit_\infty(X)$ can be identified with the fibre $F(x,y)=e^{-1}(x,y)$ of the end point evaluation map $e\colon H(V\cup X_j,V)\rightarrow V\times X_j$, $\gamma\mapsto (\gamma(0), \gamma(1))$.
\end{proposition}
\begin{proof}

Write $S:=\exit_\infty(X)$. We use the following model for the mapping space:
\begin{align*}
M(x,y)=\{x\}\times_S S^{\Delta^1_\bullet}\times_S \{y\}.
\end{align*}
That is, an $n$-simplex of $M(x,y)$ is a simplicial map $\sigma\colon(\Delta^n\times \Delta^1)_\bullet\rightarrow S$ which satisfies $\sigma(\Delta^n_\bullet\times\{0\})=\{x\}$ and $\sigma(\Delta^n_\bullet\times\{1\})=\{y\}$ (see \cite[\S 1.2.2]{LurieHTT}).

\medskip

The simplicial sets $(\operatorname{Sing} X)^{\Delta^1_\bullet}$ and $ \operatorname{Sing}(X^{|\Delta^1_\bullet|})$ are isomorphic via the adjunction $|-|\dashv \operatorname{Sing}$ and the exponential law for topological spaces. By translating the conditions on the subsimplicial sets $M(x,y)\subseteq (\operatorname{Sing} X)^{\Delta^1_\bullet}$ and $\operatorname{Sing}(F(x,y))\subseteq \operatorname{Sing} (X^{|\Delta^1_\bullet|})$ across this isomorphism, we see that it restricts to an isomorphism $M(x,y)\cong \operatorname{Sing} (F(x,y))$.
\end{proof}

\begin{remark}\label{mapping space within same stratum}
Let $X$ be a conically $I$-stratified space, let $i,j\in I$, $x\in X_i$, $y\in X_j$ and let $V$ be a neighbourhood of $x$ in $X_i$. The proposition above implies that if $i\neq j$, then $M(x,y)\subset H(V\cup X_j,V)$. If $i=j$, then $M(x,y)\cap H(V\cup X_i,V)= \emptyset$, as $V$ is a neighbourhood of $x$ in $X_i$. In this case, however, $M(x,y)$ is the mapping space in the $\infty$-category $\operatorname{Sing}(X_i)$ which can be identified with the fibre of the path space fibration of $X_i$ with respect to the basepoint $x$. Hence, $M(x,y)$ is either empty or homotopy equivalent to the loop space $\Omega(X_i,x)$.
\exend
\end{remark}

The end point evaluation map from the homotopy link is a fibration in certain situations, for example when the stratified space is metrisably conically stratified (\Cref{e fibration conically stratified space}). The following proposition simply rewrites the long exact sequence of homotopy groups arising from this fibration. To state the proposition, we need to fix some notation and various basepoints and maps --- this is done in what we for future reference will call a preamble (there is a picture below which might help to clarify the situation).

\begin{preamble}\label{preamble LES}

Let $X$ be a conically $I$-stratified space and let $i<j$ in $I$. Fix points $x_i\in X_i$, $x_j\in X_j$ and suppose there is a conical neighbourhood $U_i$ of $x_i$ in $X$ with a stratified homeomorphism $\phi_i\colon V_i\times C(L_i)\rightarrow U_i$, where $V_i$ is a weakly contractible neighbourhood of $x_i$ in $X_i$ and the union $V_i\cup X_j$ is metrisable. Suppose $M(x_i,x_j)\neq \emptyset$ and fix a path $\gamma_{ij}\in M(x_i,x_j)$; fix also an $\epsilon\in (0,1)$, for instance $\epsilon=\tfrac{1}{2}$.

\medskip

The end point evaluation map
\begin{align*}
e_{ij}\colon H(V_i\cup X_j,V_i)\rightarrow V_i\times X_j,\quad\quad \gamma\mapsto \big(\gamma(0),\gamma(1)\big)
\end{align*}
is a fibration (\Cref{e fibration conically stratified space}) and in view of \Cref{M cong F}, we may identify $M(x_i,x_j)$ with the fibre $e_{ij}^{-1}(x_i,x_j)$. Writing $L_{ij}=(L_i)_j$ for the $j$'th stratum of the link space, the map
\begin{align*}
\Psi_{ij}\colon L_{ij}\rightarrow H(V_i\cup X_j,V_i),\quad\quad
l\mapsto \big(\gamma_{x_i,l,\epsilon}\colon t\mapsto \phi_i(x_i,[l,t\epsilon])\big)
\end{align*}
is a homotopy equivalence (\Cref{link eq to htpy link}). Fix a homotopy inverse
\begin{align*}
\Psi^h_{ij}\colon H(V_i\cup X_j,V_i)\rightarrow L_{ij}
\end{align*}
and a homotopy
\begin{align*}
h\colon H(V_i\cup X_j,V_i)\times [0,1]\rightarrow H(V_i\cup X_j,V_i),\quad h\colon \operatorname{id} \sim \Psi_{ij}\circ \Psi_{ij}^h.
\end{align*}

Consider the embedding of the $j$'th link space stratum $L_{ij}$ into $X_j$
\begin{align*}
\phi_{ij}\colon L_{ij}\rightarrow X_j,\quad l\mapsto \phi_i(x_i,[l,\epsilon]).
\end{align*}

Finally, set $l_{ij}:=\Psi^h_{ij}(\gamma_{ij})\in L_{ij}$ and define a path
\begin{align*}
\eta_{ij}:=h(\gamma_{ij},-)(1)\colon [0,1]\rightarrow X_j
\end{align*}
from $x_j$ to $\phi_{ij}(l_{ij})$.

\medskip

The situation can be pictured as follows.

\begin{center}
\includegraphics[width=\textwidth]{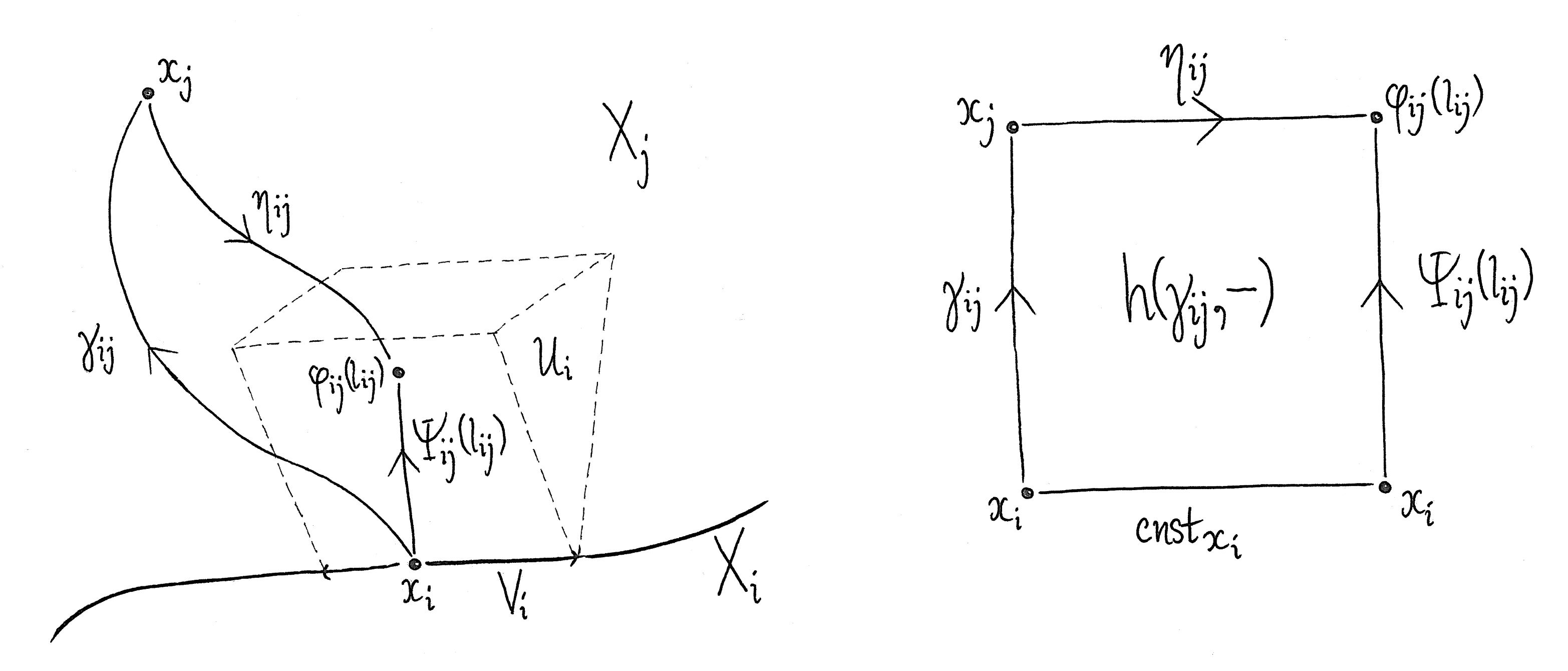}
\end{center}
\exend
\end{preamble}

\begin{proposition}\label{Homsets-LES}
In the situation of \Cref{preamble LES}, there is a long exact sequence of homotopy groups
\begin{center}
\begin{tikzpicture}
\matrix (m) [matrix of math nodes,row sep=1em,column sep=1em]
  {
    \cdots & \pi_n(L_{ij},l_{ij}) & \pi_n(X_j,x_j) & \pi_{n-1}(M(x_i,x_j), \gamma_{ij}) & \cdots & & & \\
  };
  \path[-stealth]
	(m-1-1) edge (m-1-2)
	(m-1-2) edge (m-1-3)
	(m-1-3) edge (m-1-4)
	(m-1-4) edge (m-1-5)
  ;
\end{tikzpicture}

\begin{tikzpicture}
\matrix (m) [matrix of math nodes,row sep=1em,column sep=1em]
  {
   & & & \quad\quad & \cdots & \pi_1(L_{ij},l_{ij}) & \pi_1(X_j,x_j) & \pi_{0}(M(x_i,x_j),\gamma_{ij}) & \cdots \\
  };
  \path[-stealth]
	(m-1-5) edge (m-1-6)
	(m-1-6) edge node[above]{$\phi$} (m-1-7)
	(m-1-7) edge node[above]{$\partial$} (m-1-8)
  	(m-1-8) edge (m-1-9)
  ;
\end{tikzpicture}
\end{center}

The map $\phi$ is given by conjugation by $\eta_{ij}$:
\begin{align*}
\phi\colon \pi_1(L_{ij},l_{ij})\longrightarrow \pi_1(X_j,x_j),\quad\quad [\alpha]\mapsto [\eta_{ij}^{-1}* ((\phi_{ij})_*\alpha) * \eta_{ij}],
\end{align*}
and the boundary map $\partial$ is given by concatenation with $\gamma_{ij}$:
\begin{align*}
\partial\colon \pi_1(X_j,x_j) \longrightarrow \pi_0(M(x_i,x_j),\gamma_{ij}),\quad\quad [\alpha]\mapsto [\alpha* \gamma_{ij}].
\end{align*}
\end{proposition}

\begin{proof}

We have a long exact sequence of homotopy groups arising from the fibration $e_{ij}$ in which we may replace $\pi_n(V_i,x_i)$ by $0$:
\begin{center}
\begin{tikzpicture}
\matrix (m) [matrix of math nodes,row sep=1em,column sep=1em]
  {
    \cdots & \pi_n(H(V_i\cup X_j,V_i),\gamma_{ij}) & \pi_n(X_j,x_j) & \pi_{n-1}(M(x_i,x_j), \gamma_{ij}) & \cdots \\
  };
  \path[-stealth]
	(m-1-1) edge (m-1-2)
	(m-1-2) edge (m-1-3)
	(m-1-3) edge (m-1-4)
	(m-1-4) edge (m-1-5)
  ;
\end{tikzpicture}
\end{center}

We replace $\pi_n(H(V_i\cup X_j,V_i), \gamma_{ij})$ by $\pi_n(L_{ij},l_{ij})$ via the homotopy equivalence $\Psi_{ij}$ and a basepoint change from $\Psi_{ij}(l_{ij})$ to $\gamma_{ij}$:
\begin{align*}
C_{h(\gamma_{ij},-)}\circ (\Psi_{ij})_*\colon \pi_n(L_{ij},l_{ij})\xrightarrow{\ \cong\ } \pi_n(H(V_i\cup X_j,V_i), \gamma_{ij})
\end{align*}
where $C_{h(\gamma_{ij},-)}$ denotes conjugation by the path $t\mapsto h(\gamma_{ij},t)$.

\medskip

To see that the maps are as claimed, let $\operatorname{pr}_j$ denote the projection to $X_j$ and $C_{\eta_{ij}}$ conjugation by $\eta_{ij}$. Then
\begin{align*}
\phi = (\operatorname{pr}_j\circ e_{ij})_*\circ C_{h(\gamma_{ij},-)}\circ (\Psi_{ij})_*=C_{\eta_{ij}}\circ (\phi_{ij})_*.
\end{align*}

For the boundary map $\partial$, note that it is equal to the following composite
\begin{align*}
\pi_1(X_j,x_j) \mathop{\xleftarrow{\ (\operatorname{pr}_j\circ e_{ij})_*\ }}_{\cong} \pi_1\bigg(H(V_i\cup X_j,V_i),M(x_i,x_j), \gamma_{ij}\bigg) \mathop{\longrightarrow}^{\delta} \pi_0(M(x_i,x_j), \gamma_{ij}),
\end{align*}
where the middle term is the relative homotopy group and $\delta$ is the boundary map in the long exact sequence of homotopy groups of the pair $(H(V_i\cup X_j,V_i),M(x_i,x_j))$. This is given by sending a map $f\colon [0,1]\rightarrow H(V_i\cup X_j,V_i)$ representing an element in the relative $\pi_1$ to the starting point $f(0)\in M(x_i,x_j)$. The inverse to $(\operatorname{pr}_j\circ e_{ij})_*$ is given by lifting a loop $[0,1]\rightarrow X_j$ to a path $[0,1]\rightarrow H(X_j\cup V_i,V_i)$ with end point $\gamma_{ij}$ (\cite[proof of Theorem 4.41]{Hatcher}). This is independent of the choice of lift, so for any $\alpha\colon [0,1]\rightarrow X_j$ with $\alpha(0)=\alpha(1)=x_j$, we may choose the lift $\tilde{\alpha}\colon t\mapsto \alpha|_{[0,1-t]}*\gamma_{ij}$, and we see that $\partial$ is given by $[\alpha]\mapsto[\tilde{\alpha}(0)]=[\alpha*\gamma_{ij}]$ as claimed.
\end{proof}

For mapping spaces within one stratum, \Cref{mapping space within same stratum} gives the following identification.

\begin{proposition}\label{mapping spaces within stratum homotopy groups}
Let $X$ be a conically $I$-stratified space. Let $i\in I$ and fix $x_i,x_i'\in X_i$. If $M(x_i,x_i')\neq \emptyset$, then $M(x_i,x_i')$ has the homotopy type of the loop space $\Omega (X_i,x_i)$. In particular, $\pi_n(M(x_i,x_i'),\gamma)\cong \pi_{n+1}(X_i,x_i)$ for all $n\geq 0$ and any choice of basepoint $\gamma\in M(x_i,x_i')$.
\end{proposition}

We have the following corollary.

\begin{corollary}\label{Homsets-SES}
Let $X$ be a conically $I$-stratified space. Let $i<j$, $x_i\in X_i$ and $x_j\in X_j$, and assume that the assumptions of \Cref{preamble LES} can be satisfied. Assume additionally that for any choice of $\gamma_{ij}\in M(x_i,x_j)$ in the situation of \Cref{preamble LES}, the following holds:
\begin{enumerate}[label=(\roman*)]
\item the map $\phi_{ij}\colon L_{ij}\rightarrow X_j$ is injective on $\pi_1$;
\item $\pi_n(X_j,x_j)=0$ for all $n>1$;
\item $\pi_n(L_{ij},l_{ij})=0$ for all $n>1$.
\end{enumerate}
Then the mapping space $M(x_i,x_j)$ has contractible path components and the set of path components fits into a $5$-term exact sequence
\begin{align*}
0\rightarrow \pi_1(L_{ij}, l_{ij})\xrightarrow{\phi} \pi_1(X_j,x_j)\xrightarrow{\partial} \pi_0(M(x_i,x_j))\rightarrow \pi_0(L_{ij})\rightarrow \pi_0(X_j)\rightarrow 0,
\end{align*}
where $\phi$ and $\partial$ are as described in \Cref{Homsets-LES}.

\medskip

In particular, if X is a metrisably conically stratified space admitting conical neighbourhoods with weakly contractible deepest stratum and if (i)-(iii) hold for all $i<j$, and $x_i\in X_i$ and $x_j\in X_j$ with $M(x_i,x_j)\neq \emptyset$, then the exit path $\infty$-category is equivalent to the nerve of its homotopy category $\exit_1(X)$ and the hom-sets in $\exit_1(X)$ can be identified using the exact sequences above and the isomorphisms $\pi_0(M(x_i,x_i'))\cong \pi_1(X_i,x_i)$ for $x_i,x_i'\in X_i$ in the same path component.
\end{corollary}

This result identifies (the homotopy type of) the mapping spaces in the exit path $\infty$-category, but does not tell us anything about composition. If, however, the stratified space is sufficiently contractible, then we can use \Cref{Homsets-SES} to identify the exit path $\infty$-category as in the following corollary.

\begin{corollary}\label{Exitcat=Poset}
Let $X$ be a metrisably conically $I$-stratified space with path connected, weakly contractible strata, and suppose $X$ admits conical neighbourhoods with weakly contractible strata. Then the exit path $\infty$-category $\exit_\infty(X)$ is equivalent to the nerve of its homotopy category $\exit_1(X)$ which in turn is equivalent to the poset of strata $I$.
\end{corollary}

\subsection{Group actions and exit path \texorpdfstring{$\infty$}{infinity}-categories}\label{stratified spaces and group actions}

In this section we determine the exit path $\infty$-category of stratified spaces obtained via suitably well-behaved group actions. The results should be compared with that of \cite[Theorem 1.7]{ChenLooijenga} (see \Cref{compare Chen-Looijenga}).

\medskip

A (left) action of a discrete group $G$ on a stratified space $s\colon X\rightarrow I$ consists of compatible continuous (left) actions of $G$ on $X$ and $I$, i.e.~such that the stratification map $s$ is equivariant. Recall that an action of $G$ on $X$ is \textit{properly discontinuous} if each point $x\in X$ has a neighbourhood $U$ such that the set $\{g\in G\mid g.U\cap U\neq \emptyset\}$ is finite.

\begin{remark}
A word of warning: the cones in the following theorem are the stratified cones of \Cref{stratified cone}. If $L_i$ is compact Hausdorff, then it coincides with the usual topological cone, but generally they are different. In Corollaries \ref{exit path categories and group actions II} and \ref{exit path categories and group actions II with composition} below we present a different version of this theorem in which we allow neighbourhoods in $X$ which locally look like (stratified) topological cones.
\exend
\end{remark}

\begin{theorem}\label{exit path categories and group actions}
Let $X\rightarrow I$ be a stratified space with path connected, weakly contractible strata, with $I$ satisfying the ascending chain condition, and with surjective stratification map. Suppose $G$ is a discrete group acting on $X\rightarrow I$ and let $\pi\colon X\rightarrow G\backslash X$ denote the quotient map. For any $i\in I$, denote by $G_i$ the stabiliser of $i$ and let $G_i^\ell\leq G_i$ denote the subgroup which fixes $X_i$ pointwise. Suppose that for all $i\in I$ and all $x\in X_i$ there is:
\begin{enumerate}[label=(\roman*)]
\item a $G_i^\ell$-invariant neighbourhood $U_i$ of $x$ in $X$ satisfying
\begin{align*}
\{g\in G\mid g. U_i \cap U_i\neq \emptyset \} = G_i^\ell,
\end{align*}
and such that $V_i=U_i\cap X_i$ is weakly contractible;
\item a stratified space $L_i\rightarrow I_{>i}$ with weakly contractible strata, surjective stratification map and which is equipped with with an action of $G_i^\ell$ (where the action on $I_{>i}$ is the restriction of the one of $G_i$);
\item a $G_i^\ell$-equivariant stratified homeomorphism
\begin{align*}
\phi_i\colon V_i\times C(L_i)\xrightarrow{\ \cong \ } U_i,
\end{align*}
where $G_i^\ell$ acts only on the $L_i$-coordinate of the left hand side, $g.(x,[l,t])=(x,[g.l,t])$, and such that $\phi_i$ restricts to the identity on $V_i\times \{*\}$;
\item and assume additionally that for all $j>i$, the union $V_i\cup X_j$ and its image $\pi(V_i\cup X_j)$ are metrisable.
\end{enumerate}
Then $X\rightarrow I$ and $G\backslash X\rightarrow G\backslash I$ are metrisably conically stratified spaces whose exit path $\infty$-categories are equivalent to the nerves of their homotopy categories. The exit path $1$-category of $X$ is equivalent to the poset $I$, and the exit path $1$-category of $G\backslash X$ is equivalent to the category $\mathscr{C}_{G,X}$ with objects the elements of $I$ and hom-sets
\begin{align*}
\mathscr{C}_{G,X}(i,j)=\{g\in G\mid g.i\leq j\}/G_i^\ell,
\end{align*}
where $G_i^\ell$ acts by right multiplication and with composition given by the product in $G$. Moreover, the equivalences can be chosen such that the following diagram commutes, where $\pi_*$ is the functor induced by the quotient map $\pi\colon X\rightarrow G\backslash X$ and the top vertical map sends $i\leq j$ to the morphism $i\rightarrow j$ represented by the identity element of $G$:
\begin{center}
\begin{tikzpicture}
\matrix (m) [matrix of math nodes,row sep=2em,column sep=2em]
  {
	I & \mathscr{C}_{G,X} \\
	\exit_1(X) & \exit_1(G\backslash X) \\
  };
  \path[-stealth]
  (m-1-1) edge (m-1-2) edge node[left]{$\sim$} (m-2-1)
  (m-1-2) edge node[right]{$\sim$} (m-2-2)
  (m-2-1) edge node[below]{$\pi_*$} (m-2-2)
;
\end{tikzpicture}
\end{center}
\end{theorem}

Before tackling the proof, we make some preliminary observations.

\begin{observation}
Note first of all that $G\backslash I$ is indeed a poset. The action is order preserving and we have assumed $I$ to satisfy the ascending chain condition, so we cannot have $g.i<i$ for any $g\in G$, $i\in I$. Hence, setting $[i]\leq [j]$, if $i\leq g. j$ for some $g\in G$ defines a partial order on $G\backslash I$.

\medskip

For any $i\leq j$, we have $G_j^\ell\leq G_i^\ell$ by assumption (i) and assumption (iii), since the neighbourhood $U_i$ is stratified over $I_{\geq i}$ with $j$'th stratum $U_i\cap X_j\neq \emptyset$, and $g\in G_j^\ell$ fixes $U_i\cap X_j$. This together with the fact that ${}^g G_i^\ell = G_{g.i}^\ell$ for all $i\in I$, $g\in G$, implies that the category $\mathscr{C}_{G,X}$ is well-defined.
\exend
\end{observation}

\begin{proof}
Let $i\in I$ with image $\hat{\imath}\in G\backslash I$. By assumption (i), the equivalence relations on $U_i$ induced by $G$ and $G_i^\ell$ agree, so the conical neighbourhoods in $X$ satisfying (i)-(iv) descend to $G\backslash X$:
\begin{align*}
\widehat{\phi}_i\colon V_i \times C(G_i^\ell\backslash L_i)\xrightarrow{\ \cong\ }  G_i^\ell\backslash U_i.
\end{align*}
It follows that the quotient $G\backslash X\rightarrow G\backslash I$ is indeed a conically stratified space. By (iv), both $X$ and $G\backslash X$ are metrisably conically stratified. 

\medskip

We now analyse the link spaces and show that the mapping spaces of $\exit_\infty(G\backslash X)$ have contractible components. Let $i,j\in I$ with images $\hat{\imath}, \hat{\jmath}\in G\backslash I$ and suppose $\hat{\imath}\leq \hat{\jmath}$. Write $G_{ij}=\{g\in G\mid g.i\leq j\}$. Let $x_i\in X_i$, $x_j\in X_j$ with images $\hat{x}_i$, $\hat{x}_j$ in $G\backslash X$ and let $\phi_i\colon V_i \times C(L_i)\xrightarrow{\cong} U_i $ be a conical neighbourhood of $x_i$ as in (i)-(iv). Consider the corresponding conical neighbourhood $\widehat{\phi}_i\colon V_i \times C(G_i^\ell\backslash L_i)\xrightarrow{\cong}  G_i^\ell\backslash U_i$ of $\hat{x}_i$. The $\hat{\jmath}$'th stratum of the link space $\mathscr{L}_{\hat{\imath}}:=G_i^\ell\backslash L_i$ is
\begin{align*}
\mathscr{L}_{\hat{\imath}\hat{\jmath}}=G_i^\ell \  \bigg\backslash\ \bigg(\coprod_{\overline{g}\in G_j\backslash G_{ij}}\!\! L_{i(g^{-1}.j)}\bigg)
\end{align*}
where the action of $G_j$ on $G_{ij}$ in the indexing set is given by left multiplication. For $x\in L_{i(g^{-1}.j)}$, and $u\in G_i^\ell$, we have $u.x\in L_{i(u g^{-1}.j)}= L_{i((gu^{-1})^{-1}.j)}$. Hence, the quotient map $G_j\backslash G_{ij}\rightarrow G_j\backslash G_{ij}/G_i^\ell$ induces an isomorphism
\begin{align*}
\pi_0(\mathscr{L}_{\hat{\imath}\hat{\jmath}})\cong G_j\backslash G_{ij}/G_i^\ell.
\end{align*}
Moreover, for a given $[g]\in G_j\backslash G_{ij}/G_i^\ell$, the corresponding component is homeomorphic to the quotient of $L_{i(g^{-1}.j)}$ by the action of $G_i^\ell\cap G_{g^{-1}.j}$. 

\medskip

Condition (i) implies that for all $k\in I$, the action of $G_k/G_k^\ell$ on the weakly contractible space $X_k$ is free and properly discontinuous. Hence, for any $g\in G_{ij}$, so is the action of $G_i^\ell\cap G_{g^{-1}.j}/G_{g^{-1}.j}^\ell$ on the weakly contractible space $L_{i(g^{-1}.j)}$. In fact, for any $\epsilon\in (0,1)$, the inclusion
\begin{align*}
L_{i(g^{-1}.j)}\hookrightarrow X_{g^{-1}.j},\quad l\mapsto \phi_i(x_i,[l,\epsilon]),
\end{align*}
defines a morphism of fibre bundles as below, where $\hat{X}_{\hat{\jmath}}=G\backslash (\coprod_{\overline{g}\in G_j\backslash G} X_{g^{-1}.j})$ is the $\hat{\jmath}$'th stratum of $G \backslash X$.
\begin{center}
\begin{tikzpicture}
\matrix (m) [matrix of math nodes,row sep=2em,column sep=1em]
  {
	(G_i^\ell\cap G_{g^{-1}.j})/G_{g^{-1}.j}^\ell & G_{g^{-1}.j}/G_{g^{-1}.j}^\ell \\
	L_{i(g^{-1}.j)} & X_{g^{-1}.j} \\
	\mathscr{L}_{\hat{\imath}\hat{\jmath}} & \hat{X}_{\hat{\jmath}} \\
  };
  \path[right hook-stealth]
  	(m-1-1) edge (m-2-1)
  	(m-1-2) edge (m-2-2)
  	(m-1-1) edge (m-1-2)
	(m-2-1) edge (m-2-2)
	(m-3-1) edge (m-3-2)
  ;
  \path[->>]
  	(m-2-1) edge (m-3-1)
  	(m-2-2) edge (m-3-2)
  ;
\end{tikzpicture}
\end{center}
It follows that for any choice of basepoint in the $[g]$'th component of $\mathscr{L}_{\hat{\imath}\hat{\jmath}}$, the embedding $\mathscr{L}_{\hat{\imath}\hat{\jmath}} \rightarrow \hat{X}_{\hat{\jmath}}$ induces the inclusion
\begin{align*}
(G_i^\ell\cap G_{g^{-1}.j})/G_{g^{-1}.j}^\ell \hookrightarrow G_{g^{-1}.j}/G_{g^{-1}.j}^\ell
\end{align*}
on $\pi_1$ and the higher homotopy groups of both $\mathscr{L}_{\hat{\imath}\hat{\jmath}}$ and $\hat{X}_{\hat{\jmath}}$ vanish. By \Cref{Homsets-SES}, the mapping space $M(\hat{x}_i,\hat{x}_j)$ of $\exit_\infty(G\backslash X)$ has contractible components. As this holds for any choice of $i<j$ and any points $\hat{x}_i$, $\hat{x}_j$ and since the strata are Eilenberg-Maclane spaces, the exit path $\infty$-category $\exit_\infty(G\backslash X)$ is equivalent to the nerve of its homotopy category $\exit_1(G\backslash X)$.

\medskip

We now define a functor $\mathscr{C}_{G,X}\rightarrow \exit_1(G\backslash X)$ and show that this is an equivalence by examining the exact sequences of \Cref{Homsets-SES} in more detail. Denote by $\pi_*\colon \exit_1(X)\rightarrow \exit_1(G\backslash X)$ the map induced by $\pi$. By \Cref{Exitcat=Poset}, the exit path $\infty$-category of $X$ is equivalent to the nerve of its homotopy category $\exit_1(X)$ which in turn is equivalent to the poset of strata $I$. In view of this, if $x,x'\in X$ are connected by a morphism in $\exit_1(X)$, then this morphism is unique and we denote it by $p_{x\rightarrow x'}$. Choose basepoints $x_i\in X_i$ for all $i\in I$ and define a functor $F\colon \mathscr{C}_{G,X}\rightarrow \exit_1(G\backslash X)$ as follows
\begin{align*}
F(i)=\hat{x}_i=\pi(x_i),\quad \text{and} \quad F([g]\colon i\rightarrow j)=\pi_*(p_{x_i\rightarrow g^{-1}.x_j}).
\end{align*}
This is well-defined, since for $u\in G_i^\ell$, the path $p_{x_i\rightarrow u^{-1}.x_i}$ is the trivial loop at $x_i$.

\medskip

To see that $F$ is fully faithful, let $i\neq j$ and $g_{ij}\in G_{ij}$ (if $G_{ij}=\emptyset$, then the hom-sets on both sides are empty and there is nothing to prove). Set $\gamma_{ij}=F([g_{ij}])\in M(\hat{x}_i,\hat{x}_j)$ and fix, according to \Cref{preamble LES}, a compatible basepoint $l_{ij}\in \mathscr{L}_{\hat{\imath}\hat{\jmath}}$. Then we have a commutative diagram of exact sequences as below, and as this holds for any choice of $g_{ij}$, $F$ is bijective on hom-sets for $i\neq j$ by an extended 5-lemma (see for example \cite[\S 4.1 Exercise 9]{Hatcher}).

\begin{center}
\begin{tikzpicture}
\matrix (m) [matrix of math nodes,row sep=2em,column sep=1em]
  {
    0 & (G_i^\ell\cap G_{g_{ij}^{-1}.j})/G_{g_{ij}^{-1}.j}^\ell & G_j/G_j^\ell & \mathscr{C}_{G,X}(i,j) & G_j\backslash G_{ij}/ G_i^\ell & 0 \\
    0 & \pi_1(\mathscr{L}_{\hat{\imath}\hat{\jmath}},l_{ij}) & \pi_1(\hat{X}_{\hat{\jmath}},\hat{x}_j) & \pi_0(M(\hat{x}_i,\hat{x}_j),\gamma_{ij}) & \pi_0(\mathscr{L}_{\hat{\imath}\hat{\jmath}},l_{ij}) & 0 \\
  };
  \path[-stealth]
	(m-1-1) edge (m-1-2)
	(m-1-2) edge (m-1-3) edge node[right]{$\cong$} (m-2-2)
	(m-1-3) edge node[above]{$-\cdot g_{ij}$} (m-1-4) edge node[right]{$\cong$} (m-2-3)
	(m-1-4) edge (m-1-5) edge node[right]{$F$} (m-2-4)
	(m-1-5) edge (m-1-6) edge node[right]{$\cong$} (m-2-5)
	(m-2-1) edge (m-2-2)
	(m-2-2) edge node[below]{$\phi$} (m-2-3)
	(m-2-3) edge node[below]{$\partial$} (m-2-4)
	(m-2-4) edge (m-2-5)	
	(m-2-5) edge (m-2-6)
  ;
\end{tikzpicture}
\end{center}

For $i=j$, bijectivity on hom-sets follows from the fact that $\hat{X}_{\hat{\imath}}$ is a $K(G_i/G_i^\ell,1)$. The functor is essentially surjective since the strata are path connected, so we have established the desired equivalence. If, in addition to the functor $F$, we choose the equivalence $I\xrightarrow{\sim} \exit_1(X)$ which sends $i$ to $x_i$ and $i<j$ to $p_{x_i\rightarrow x_j}$, then the diagram in the statement of the theorem commutes.
\end{proof}

\begin{remark}
Note that the equivalence $F$ defined in the proof of \Cref{exit path categories and group actions} mirrors the identification of the fundamental group of a (nice) topological space with the group of deck transformations of its universal cover: if $\pi\colon \tilde{X}\rightarrow X$ is a universal cover of a space $X$, where $\tilde{X}$ (and thus $X$) is locally path connected, with basepoints $\tilde{x}\in \tilde{X}$ and $x=\pi(\tilde{x})\in X$, and if $G$ is the group of deck transformations, then we may define a group isomorphism $G\rightarrow \pi_1(X,x)$ by sending $g$ to (the homotopy class of) the loop $\pi\circ p_g$, where $p_g$ is the unique (up to homotopy) path in $\tilde{X}$ from $\tilde{x}$ to $g^{-1}.\tilde{x}$.

\medskip

In this respect, the map $\pi\colon X\rightarrow G\backslash X$ can be interpreted as a stratified universal cover --- see \cite{Woolf} for more about stratified covers (of homotopically stratified sets).
\exend
\end{remark}

We wish to extend \Cref{exit path categories and group actions} to a larger class of stratified spaces. \Cref{exit path categories and group actions II} and \Cref{exit path categories and group actions II with composition} below provide a version of this result applicable to the case when $X$ is not necessarily conically stratified, but does admit neighbourhoods which are ``conical'' with respect to the stratified topological cone defined below. In particular, it will apply to the reductive Borel--Serre compactification as we will see in \Cref{exit category of rbs}. The proofs are a slight modification of that of \Cref{exit path categories and group actions}.

\begin{definition}\label{stratified topological cone}
Let $s\colon Y\rightarrow I$ be a stratified space. The \textit{(open) stratified topological cone} on $Y$ is the stratified space $s^{\triangleleft}\colon C^t(Y)\rightarrow I^{\triangleleft}$ defined as follows: the poset of strata is $I^{\triangleleft}:=I\cup \{-\infty\}$ with $-\infty\leq i$ for all $i\in I$, and $C^t(Y)=(Y\times [0,1))\cup_{Y\times \{0\}}\ast$ is the (usual topological) pushout, and the stratification map is given by $s^\triangleleft(x,t)=s(x)$ and $s^\triangleleft(\ast)=-\infty$.
\defend
\end{definition}

\begin{remark}
As remarked below \Cref{stratified cone}, the stratified topological cone $C^t(Y)$ agrees with the stratified cone $C(Y)$, when $Y$ is compact Hausdorff.
\exend
\end{remark}

The theorem below is a version of \Cref{exit path categories and group actions} for ``topologically conically stratified'' spaces. It applies to a larger class of stratified spaces, but the conclusion is weaker, as it does not provide information about composition in the exit path category. For this we need additional data as described in \Cref{exit path categories and group actions II with composition} and \Cref{remark about weak topology}.

\begin{theorem}\label{exit path categories and group actions II}
Let $X\rightarrow I$ be a stratified space with path connected, weakly contractible strata, with $I$ satisfying the ascending chain condition, and with surjective stratification map. Suppose $G$ is a discrete group acting on $X\rightarrow I$ and let $\pi\colon X\rightarrow G\backslash X$ denote the quotient map. Let for all $i\in I$, $G_i$ denote the stabiliser of $i$ and let $G_i^\ell\leq G_i$ denote the subgroup which fixes $X_i$ pointwise. Suppose that for all $i\in I$ and all $x\in X_i$ there is:
\begin{enumerate}[label=(\roman*)]
\item a $G_i^\ell$-invariant neighbourhood $U_i$ of $x$ in $X$ satisfying
\begin{align*}
\{g\in G\mid g. U_i \cap U_i\neq \emptyset \} =G_i^\ell,
\end{align*}
and such that $V_i=U_i\cap X_i$ is weakly contractible;
\item a stratified space $L_i\rightarrow I_{>i}$ with weakly contractible strata, surjective stratification map, and which is equipped with an action of $G_i^\ell$ such that the quotient $G_i^\ell\backslash L_i$ is compact Hausdorff (where the action on $I_{>i}$ is the restriction of the one of $G_i$);
\item a $G_i^\ell$-equivariant stratified homeomorphism
\begin{align*}
\phi_i\colon V_i\times C^t(L_i)\xrightarrow{\ \cong\ } U_i,
\end{align*}
where $G_i$ acts only on the $L_i$-coordinate of the left hand side, $g.(x,[l,t])=(x,[g.l,t])$, and such that $\phi_i$ restricts to the identity on $V_i\times \{*\}$;
\item and assume additionally that the image $\pi(V_i\cup X_j)$ is metrisable for all $j>i$.
\end{enumerate}
Then $G\backslash X\rightarrow G\backslash I$ is a metrisably conically stratified space, and the exit path $\infty$-category of $G\backslash X$ is equivalent to the nerve of its homotopy category $\exit_1(G\backslash X)$. The exit path $1$-category is equivalent to a category with objects the elements of $I$ and hom-sets
\begin{align*}
\operatorname{Hom}(i,j)=\{g\in G\mid g.i\leq j\}/G_i^\ell,
\end{align*}
where $G_i^\ell$ acts by right multiplication.
\end{theorem}
\begin{proof}
As in \Cref{exit path categories and group actions}, $G\backslash I$ has a natural partial order. The maps $\phi_i$ descend to the quotient
\begin{align*}
\widehat{\phi}_i\colon V_i \times C^t(G_i^\ell\backslash L_i)\xrightarrow{\ \cong \ }  G_i^\ell\backslash U_i,
\end{align*}
and since $G_i^\ell\backslash L_i$ is assumed to be compact Hausdorff, the stratified topological cone $C^t(G_i^\ell\backslash L_i)$ coincides with $C(G_i^\ell\backslash L_i)$. We thus conclude that $G\backslash X\rightarrow G\backslash I$ is a metrisably conically stratified space. The proof of \Cref{exit path categories and group actions} goes through word for word up until defining the functor $F$, so we conclude that $\exit_\infty (G\backslash X)$ is equivalent to the nerve of its homotopy category.

\medskip

Since the strata of $X$ are path connected, we can fix basepoints $x_i\in X_i$ for all $i$ and any object of $\exit_1(G\backslash X)$ will be equivalent to $\hat{x}_i=\pi(x_i)$ for some $i$. To determine the hom-sets in $\exit_1(G\backslash X)$, we choose a path $\gamma_{ij}\in M(\hat{x}_i,\hat{x}_j)$ for $i\neq j$ (if the mapping space is empty, there is nothing to prove), and the identification of the set of path components as in the proof of \Cref{exit path categories and group actions} follows through word for word. As does the identification for $i=j$, since $\hat{X}_{\hat{\imath}}$ is again a $K(G_i/G_i^\ell,1)$.
\end{proof}

What is lacking in the above situation is a collection of compatible exit paths that allow us to also analyse the composition in $\exit_1(G\backslash X)$. If such a collection exists, then we can fully identify the exit path $1$-category as in the following theorem (see also \Cref{remark about weak topology}).

\begin{proposition}\label{exit path categories and group actions II with composition}
Suppose that in the situation of \Cref{exit path categories and group actions II}, we can choose basepoints $x_i\in X_i$ for all $i\in I$ and paths $\gamma_{ij}^g\colon [0,1]\rightarrow X$ with $\gamma_{ij}^g(0)=x_i$ and $\gamma_{ij}^g(1)=g^{-1}.x_j$ for all $i,j\in I$ and $g\in G$ with $g.i\leq j$. Assume that these paths satisfy the following conditions:
\begin{enumerate}[label=(\roman*)]
\item $\gamma_{ij}^g\in H(X_i\cup X_{g^{-1}.j},X_i)$, when $g.i<j$;
\item $\gamma_{ij}^g\in C([0,1],X_i)$, when $g.i=j$;
\item $\gamma_{ii}^u$ is the constant loop at $x_i$ for all $u\in G_i^\ell$;
\item the concatenations are functorial: for all $i,j,k\in I,$ and $g,h\in G$ with $g.i\leq j$, $h.j\leq k$, we have equalities in $\exit_1(G\backslash X)$:
\begin{align*}
(\pi\circ\gamma_{jk}^h)\ast (\pi\circ\gamma_{ij}^g) = (\pi\circ\gamma_{ik}^{hg}).
\end{align*}
\end{enumerate}
Then there is a functor $F\colon \mathscr{C}_{G,X}\rightarrow \exit_1(G\backslash X)$, $F(i)=\pi(x_i)$, $F([g]\colon i\rightarrow j)=\pi\circ \gamma_{ij}^g$, where $\mathscr{C}_{G,X}$ is the category with objects the elements of $I$ and hom-sets
\begin{align*}
\mathscr{C}_{G,X}(i,j)=\{g\in G\mid g.i\leq j\}/G_i^\ell,
\end{align*}
where $G_i^\ell$ acts by right multiplication, and with composition given by the product in $G$. Moreover, the functor $F$ is an equivalence.
\end{proposition}
\begin{proof}
As in \Cref{exit path categories and group actions}, the category $\mathscr{C}_{G,X}$ is well-defined, and conditions (i)-(iv) imply that $F$ is well-defined. It fits into the proof of \Cref{exit path categories and group actions II} where it is seen to be an equivalence.
\end{proof}

\begin{remark}\label{remark about weak topology}
In the situation of \Cref{exit path categories and group actions II}, one can weaken the topology of $X$ in order to obtain a conically stratified space with the same quotient space $G\backslash X$. This is similar to the trick used by Milnor to construct universal bundles in \cite{Milnor56}. More specifically, let $X^w$ denote the space whose underlying set is that of $X$ equipped with the coarsest topology such that
\begin{itemize}[label=$\ast$]
\item the stratification map $X^w\rightarrow I$ is continuous;
\item the map $X^w\rightarrow G\backslash X$ is a quotient map;
\item the inclusions $X_i\hookrightarrow X^w$, $i\in I$, are embeddings.
\end{itemize}

Suppose we have a stratified space $L\rightarrow I$ equipped with an action of $G$ and let $G$ act on the stratified topological cone $C^t(L)\rightarrow I^\triangleleft$ by acting on the $L$-coordinate of the cone and fixing the apex. If $G\backslash L$ is compact Hausdorff, then $C^t(L)^w\cong C(L^w)$. Hence, if $X$ admits neighbourhoods as in \Cref{exit path categories and group actions II}, then $X^w$ is a conically stratified space with quotient space $G\backslash X$ and we are almost in the situation of \Cref{exit path categories and group actions}. However, the metrisability conditions that need to be verified for \Cref{exit path categories and group actions} to apply, mean that we cannot make a reasonable general statement using $X^w$ as an intermediary construction. At least not with the tools at hand.

\medskip

In particular, we would expect that in most concrete cases of \Cref{exit path categories and group actions II}, the composition rule in the exit path $1$-category $\exit_1(G\backslash X)$ does coincide with the natural one of $\mathscr{C}_{G,X}$ given by multiplication in $G$. This composition rule is well-defined and we do not have any concrete counter examples to such a claim. The problem is solely that in the situation of \Cref{exit path categories and group actions II}, we are unable to make an explicit comparison between the two composition rules --- and of course we cannot rule out that our lack of control of the finer technicalities of the situation may result in a counter example.
\exend
\end{remark}

\begin{observation}\label{restricting to full subcategories}
Suppose we are in the situation of either \Cref{exit path categories and group actions} or \Cref{exit path categories and group actions II with composition} and consider for some subset $J\subseteq I$, the image of the union $\bigcup_{j\in J}X_j$ under the quotient map $X\rightarrow G\backslash X$:
\begin{align*}
\bigcup_{j\in J}\hat{X}_{\hat{\jmath}}\subseteq G\backslash X.
\end{align*}
Then the exit path $\infty$-category of $\bigcup_{j\in J}\hat{X}_{\hat{\jmath}}$ is also equivalent to the nerve of its homotopy category and the functor $\mathscr{C}_{G,X}\rightarrow \exit_1(G\backslash X)$ defined in the proof of \Cref{exit path categories and group actions} or \Cref{exit path categories and group actions II with composition} restricts to an equivalence
\begin{align*}
\mathscr{C}_{G,X}(J)\longrightarrow \exit_1(\textstyle\bigcup_{j\in J}\hat{X}_{\hat{\jmath}}).
\end{align*}
of the full subcategory $\mathscr{C}_{G,X}(J)$ spanned by the elements of $J$ and the exit path $1$-category of $\bigcup_{j\in J}\hat{X}_{\hat{\jmath}}$.
\exend
\end{observation}

The following corollary recovers the result of \cite[Theorem 1.7]{ChenLooijenga} in the cases to which their theorem also applies (see also \Cref{compare Chen-Looijenga}).

\begin{corollary}\label{homotopy type and group actions}
If in the situation of \Cref{exit path categories and group actions} or \Cref{exit path categories and group actions II with composition}, the space $G\backslash X$ is paracompact and locally contractible, then it is weakly homotopy equivalent to the geometric realisation of the category $\mathscr{C}_{G,X}$. Moreover, this equivalence is functorial with respect to inclusions of unions of strata in the sense that the restriction functors of \Cref{restricting to full subcategories} also induce weak homotopy equivalences.
\end{corollary}
\begin{proof}
We have a zig-zag of functors of $\infty$-categories,
\begin{align*}
N(\mathscr{C}_{G,X})\ \longrightarrow\ N(\exit_1(G\backslash X))\ \longleftarrow\ \exit_\infty(G\backslash X)\ \longrightarrow\ \operatorname{Sing}(G\backslash X),
\end{align*}
where the first is given by the equivalence $F\colon \mathscr{C}_{G, X}\rightarrow \exit_1(G\backslash X)$, the second is the canonical map from $\exit_\infty(G\backslash X)$ to the nerve of its homotopy category, and the third is the weak homotopy equivalence of \Cref{exit and sing weak htpy equiv}. Functoriality with respect to inclusions of unions of strata follows directly from \Cref{restricting to full subcategories}.
\end{proof}

We also have the following corollary, which is the case when all the subgroups $G_i^\ell$ are trivial.

\begin{corollary}\label{exit path categories and group actions with trivial stabilisers}
Let $X\rightarrow I$ be a metrisably conically stratified space with $I$ satisfying the ascending chain condition. Suppose the strata of $X$ are path connected and weakly contractible and that we can choose conical neighbourhoods with weakly contractible strata. Let $G$ be a discrete group acting on $X\rightarrow I$ with metrisable quotient space $G\backslash X$. If the action of $G$ on $X$ is free and properly discontinuous, then $G\backslash X\rightarrow G\backslash I$ is a metrisably conically stratified space. Moreover, the exit path category of $G\backslash X$ is equivalent to the nerve of its homotopy category which is equivalent to the category $\mathscr{C}_{G,X}$ with objects the elements of $I$, hom-sets
\begin{align*}
\mathscr{C}_{G,X}(i,j)=\{g\in G\mid g.i\leq j\},
\end{align*}
and composition given by the product in $G$. The equivalence can be chosen such that it is functorial with respect to inclusions of unions of strata. If moreover, $G\backslash X$ is locally contractible, then it is weakly homotopy equivalent to the geometric realisation $|\mathscr{C}_{G,X}|$, and this weak homotopy equivalence is also functorial with respect to inclusions of unions of strata.
\end{corollary}

\begin{remark}\label{compare Chen-Looijenga}
The results of this section should be compared with \cite[Theorem 1.7]{ChenLooijenga}. We rephrase and slightly strengthen their result in certain situations. The settings differ a great deal; for one we refrain from talking about stacks, orbifolds and orbispaces in this paper, and the conditions on the space $X$ in Theorems \ref{exit path categories and group actions} , \ref{exit path categories and group actions II} and \ref{exit path categories and group actions II with composition} are much more restrictive. In particular, we are only able to compare with (a subset of) the situations in which \cite[Theorem 1.7]{ChenLooijenga} gives an actual homotopy equivalence and not the more general \textit{stacky} homotopy equivalence. In the comparable cases, however, we strengthen their result by determining not just a homotopy type which is functorial with respect to inclusions of unions of strata (\Cref{homotopy type and group actions}), but the exit path $\infty$-category. Thus our result completely captures constructible sheaves on the stratified spaces in question, and as a corollary determines the homotopy type, while the result of Chen--Looijenga determines the homotopy type but fails to recover the constructible sheaves (see also a more concrete example and explanation of the difference below). We also believe that the conditions needed for our result to apply may be somewhat easier to check, since they are local in nature, whereas the theorem of Chen--Looijenga requires compatible well-behaved neighbourhoods of each stratum.

\medskip

The difference between the exit path $\infty$-category and a homotopy equivalence which is functorial with respect to inclusions of unions of strata can be summed up as the difference between considering a space and its suspension. The mapping spaces of the exit path $\infty$-category keep track of the glueing data and we may lose this information by only considering the homotopy type. Consider the following example: let $X$ be a finite CW-complex with $\pi_1(X)\neq 0$ and trivial homology, e.g.~the $2$-skeleton of the Poincaré homology sphere. Then the (unreduced) suspension $SX$ is contractible, since it is simply connected and $SX\rightarrow *$ is a homology isomorphism. Stratify $SX$ over $\{0<1\}$ by sending $[x,0]$ to $0$ and $[x,t]$ to $1$ for $t>0$ --- this is conically stratified, as $X$ is compact Hausdorff. The map $SX\rightarrow |[1]|\cong [0,1]$, $[x,t]\mapsto t$, is a homotopy equivalence which is functorial with respect to inclusions of unions of strata. However, $X$ is a link space of the point $[x,0]$ in $SX$, so an application of \Cref{Homsets-LES} shows that the exit path $\infty$-category of $SX$ is equivalent to the topological category $\mathscr{C}$ with objects $0$ and $1$ and morphism spaces $\mathscr{C}(i,i)\simeq *$, $i=0,1$, and $\mathscr{C}(0,1)\simeq X$.
\exend
\end{remark}

\section{The reductive Borel--Serre compactification}\label{rbs}

We introduce the Borel--Serre and reductive Borel--Serre compactifications of a locally symmetric space $\Gamma\backslash X$ associated with an arithmetic group $\Gamma$. Zucker's original definition of the reductive Borel--Serre compactification is as a quotient of the Borel--Serre compactification. Our construction is slightly different and follows the presentation of \cite{JiMurtySaperScherk}: it will be the quotient of a suitable stratified space under an action of $\Gamma$ allowing us to apply the calculational tools developed in the previous section. We refer the reader to the original papers \cite{BorelSerre} and \cite{Zucker82} and also to \cite{GoreskyHarderMacPherson}, \cite{BorelJi} and \cite{JiMurtySaperScherk} for more details. In \Cref{exit category of rbs} we interpret these spaces as poset-stratified spaces and determine their exit path $\infty$-categories, proving the main result of this paper.

\subsection{Locally symmetric spaces and compactifications}\label{rbs definition}

Let $\mathbf{G}$ be a connected reductive linear algebraic group defined over $\Q$ and let $\Gamma\leq \mathbf{G}(\Q)$ be an arithmetic group. We will assume that the centre of $\mathbf{G}$ is anisotropic over $\Q$, i.e.~of $\Q$-rank $0$ (see \cite[\S 3]{GoreskyHarderMacPherson}, \cite[\S 5.0]{BorelSerre} for details on why we can reduce to this case). Choose a maximal compact subgroup $K\leq G=\mathbf{G}(\R)$ and consider the symmetric space $X\cong G/K$ of maximal compact subgroups of $G$ and denote by $x_0\in X$ the basepoint corresponding to $K$. The space $X$ is homeomorphic to Euclidean space and the action of $\Gamma$ by left multiplication is properly discontinuous. If $\Gamma$ is torsion-free then the quotient $\Gamma\backslash X$ is a locally symmetric space which by the Godement compactness criterion is compact if and only if $\mathbf{G}$ has $\Q$-rank $0$ (\cite[III.2.15]{BorelJi}). From now on we assume that $\mathbf{G}$ has positive $\Q$-rank. We will need to assume that $\Gamma$ is neat later on: recall that a subgroup $H\subset \mathbf{G}(\Q)$ is \textit{neat}, if for some (hence any) faithful representation $\rho\colon \mathbf{G}\rightarrow\mathbf{GL}_n$ over $\Q$, the subgroup of $\C^\times$ generated by the eigenvalues of $\rho(h)$ is torsion-free for all $h\in H$. If $H$ is neat, then it is torsion-free. Any arithmetic group contains a finite index neat subgroup (\cite[\S 17.6]{Borel69}).

\medskip

Given a rational parabolic subgroup $\mathbf{P}\leq \mathbf{G}$, denote by $\mathbf{N_P}\leq \mathbf{P}$ the unipotent radical of $\mathbf{P}$ and by $\mathbf{L_P}=\mathbf{P}/\mathbf{N_P}$ the Levi quotient. Let $\mathbf{S_P}$ denote the maximal $\Q$-split torus in the centre of $\mathbf{L_P}$, and let $\mathbf{M_P}=\bigcap_\chi \ker\chi^2$ denote the intersection of the kernels of the squares of all rationally defined characters on $\mathbf{L_P}$. Write $A_{\mathbf{P}}=\mathbf{S_P}(\R)^0$ for the identity component of the real points of $\mathbf{S_P}$, and $M_{\mathbf{P}}=\mathbf{M_P}(\R)$. Then the real points $L_{\mathbf{P}}=\mathbf{L_P}(\R)$ has a direct sum decomposition $L_{\mathbf{P}}=A_{\mathbf{P}}\times M_{\mathbf{P}}$, which induces the \textit{rational Langlands decomposition} of $P=\mathbf{P}(\R)$:
\begin{align*}
P\cong  N_{\mathbf{P}}\times A_{\mathbf{P},x_0} \times M_{\mathbf{P},x_0},
\end{align*} 
where $N_{\mathbf{P}}=\mathbf{N_P}(\R)$, and $A_{\mathbf{P},x_0}$ and $M_{\mathbf{P},x_0}$ are the lifts of $A_{\mathbf{P}}$ and $M_{\mathbf{P}}$ to the unique Levi subgroup which is stable under the extended Cartan involution of $G$ associated with $K$. Since $G=PK$, $P$ acts transitively on $X$, and so the Langlands decomposition of $P$ gives rise to the \textit{horospherical decomposition} of $X$
\begin{align*}
X\cong N_{\mathbf{P}}\times A_{\mathbf{P},x_0}\times X_{\mathbf{P},x_0},
\end{align*}
where $X_{\mathbf{P},x_0}$ is the symmetric space
\begin{align*}
X_{\mathbf{P},x_0}=M_{\mathbf{P},x_0}/(M_{\mathbf{P},x_0}\cap K)\cong L_{\mathbf{P}}/A_{\mathbf{P}}K_{\mathbf{P}},
\end{align*}
where $K_{\mathbf{P}}\leq M_{\mathbf{P}}$ corresponds to $M_{\mathbf{P},x_0}\cap K$. We define the \textit{geodesic action} of $A_{\mathbf{P}}$ on $X$ by identifying $A_{\mathbf{P}}$ with the lift $A_{\mathbf{P},x_0}$ and letting it act on $X$ by the translation action on the $A_{\mathbf{P},x_0}$-factor of the horospherical decomposition. This action turns out to be independent of the basepoint $x_0$ (\cite[\S 3.2]{BorelSerre}). From now on we omit the reference to the basepoint $x_0$.

\medskip

For every rational parabolic subgroup $\mathbf{P}$ of $\mathbf{G}$ define the \textit{Borel--Serre boundary component} as
\begin{align*}
e(\mathbf{P})=N_{\mathbf{P}}\times X_{\mathbf{P}} \cong X/A_{\mathbf{P}}
\end{align*}
according to the decompositions above. As a set, we define the \textit{partial Borel--Serre compactification} as the disjoint union of the Borel--Serre boundary components for all rational parabolic subgroups of $\mathbf{G}$.
\begin{align*}
\XBS := \coprod_{\mathbf{P}} e(\mathbf{P}),
\end{align*}
where we also interpret $\mathbf{G}$ as a parabolic subgroup of itself with $e(\mathbf{G})=X_{\mathbf{G}}=X$.

\medskip

To define the topology on $\XBS$, we define corners associated with the rational parabolic subgroups which simultaneously equip $\XBS$ with the structure of a manifold with corners. Let $\mathbf{P}$ be a rational parabolic subgroup of $\mathbf{G}$ and let $\Delta_{\mathbf{P}}$ denote the simple roots of $P$ with respect to $A_{\mathbf{P}}$ (\cite[III.1.13]{BorelJi}). We denote the value of a character $\alpha$ on $a\in A_{\mathbf{P}}$ by $a^\alpha$ and make the identification
\begin{align*}
A_{\mathbf{P}}\xrightarrow{\cong} (\R_{>0})^{\Delta_{\mathbf{P}}},\quad a\mapsto (a^{-\alpha})_{\alpha\in \Delta_{\mathbf{P}}}.
\end{align*}
The closure of $A_{\mathbf{P}}$ in $\R^{\Delta_{\mathbf{P}}}$ is $(\R_{\geq 0})^{\Delta_{\mathbf{P}}}$ and we denote this by $\overline{A_{\mathbf{P}}}$ --- for clarity, we use the character notation to denote the coordinates, i.e. the $\alpha$'th coordinate of $a\in \overline{A_{\mathbf{P}}}$ is denoted $a^{-\alpha}$. There is an inclusion preserving bijective correspondence between rational parabolic subgroups containing $\mathbf{P}$ and subsets of $\Delta_{\mathbf{P}}$, with $P$ corresponding $\emptyset$ and $G$ to $\Delta_{\mathbf{P}}$ (\cite[III.1.15]{BorelJi}). Denote by $\Delta_{\mathbf{P}}^{\mathbf{Q}}\subseteq \Delta_{\mathbf{P}}$ the subset corresponding to $\mathbf{Q}\geq \mathbf{P}$ and define a point $o_{\mathbf{Q}}\in \overline{A_{\mathbf{P}}}$ with coordinates $o_{\mathbf{Q}}^{-\alpha}=1$ for all $\alpha\in \Delta_{\mathbf{P}}^{\mathbf{Q}}$ and $o_{\mathbf{Q}}^{-\alpha}=0$ for $\alpha\notin \Delta_{\mathbf{P}}^{\mathbf{Q}}$. Note that $o_{\mathbf{P}}$ corresponds to the origin in $\R^{\Delta_{\mathbf{P}}}$ and that $o_{\mathbf{G}}$ corresponds to the point all of whose coordinates are $1$. For every rational parabolic subgroup $\mathbf{P}$ of $\mathbf{G}$ define the \textit{corner associated with $\mathbf{P}$} as
\begin{align*}
X(\mathbf{P}):=\overline{A_{\mathbf{P}}}\times e(\mathbf{P})=\overline{A_{\mathbf{P}}}\times N_{\mathbf{P}}\times X_{\mathbf{P}},
\end{align*}
and for $\mathbf{Q}\geq \mathbf{P}$, identify $e(\mathbf{Q})$ with $(A_{\mathbf{P}}\cdot o_{\mathbf{Q}})\times N_{\mathbf{P}}\times X_{\mathbf{P}}$ (\cite[III.5.6]{BorelJi}), where $A_{\mathbf{P}}\cdot o_{\mathbf{Q}}$ is the orbit of $o_{\mathbf{Q}}$ under the action of $A_{\mathbf{P}}$ on $\overline{A_{\mathbf{P}}}$ by coordinatewise multiplication, i.e. the subspace spanned by the coordinates $\Delta_{\mathbf{P}}^{\mathbf{Q}}\subseteq \Delta_{\mathbf{P}}$ (all other coordinates are zero). In particular, we identify $e(\mathbf{P})$ with $\{o_{\mathbf{P}}\}\times N_{\mathbf{P}}\times X_{\mathbf{P}}$ and $X$ with $A_{\mathbf{P}}\times N_{\mathbf{P}}\times X_{\mathbf{P}}$. Equip $\XBS$ with the finest topology such that for all rational parabolic subgroups $\mathbf{P}$ of $\mathbf{G}$, the inclusion
\begin{align*}
X(\mathbf{P})\cong\coprod_{\mathbf{Q}\geq \mathbf{P}} e(\mathbf{Q})\ \longrightarrow\ \coprod_{\mathbf{Q}} e(\mathbf{Q})=\XBS
\end{align*}
is an open embedding. This equips $\XBS$ with the structure of a real analytic manifold with corners which is Hausdorff and paracompact (\cite[Theorem 7.8]{BorelSerre}). The original space $X=X(\mathbf{G})$ is identified with the interior of $\XBS$; in particular, the embedding $X \hookrightarrow \XBS$ is a homotopy equivalence, so $\XBS$ is also contractible.

\medskip

For each rational parabolic subgroup $\mathbf{P}$, the action of $P$ on $X$ descends to an action on the boundary component $e(\mathbf{P})$. The action of $\mathbf{G}(\Q)$ on $X$ can be extended to an action on $\XBS$ which permutes the boundary components, $g.e(\mathbf{P})=e({}^g\mathbf{P})$, and which restricts to the action of $\mathbf{P}(\Q)$ on $e(\mathbf{P})$ (\cite[Proposition 7.6]{BorelSerre}, \cite[III.5.13]{BorelJi}).

\medskip

The action of $\Gamma$ on $\XBS$ is properly discontinuous and the quotient $\Gamma\backslash \XBS$ is compact Hausdorff (\cite[Theorem 9.3]{BorelSerre}, \cite[III.5.14]{BorelJi}). If $\Gamma$ is torsion-free, then the action is free and the quotient map $\XBS\rightarrow \Gamma\backslash \XBS$ is a local homeomorphism. In particular, $\Gamma\backslash \XBS$ is also a manifold with corners, and $\Gamma \backslash X$ identifies with the interior of $\Gamma\backslash \XBS$. Thus, the embedding $Gamma \backslash X\hookrightarrow\Gamma\backslash\XBS$ is a homotopy equivalence and both spaces are models for the classifying space of $\Gamma$.

\medskip

We move on to define the reductive Borel--Serre compactification. For every rational parabolic subgroup $\mathbf{P}$ of $\mathbf{G}$ define the \textit{reductive Borel--Serre boundary component} as
\begin{align*}
\hat{e}(\mathbf{P})=X_{\mathbf{P}}
\end{align*}
so that we have a projection $e(\mathbf{P})\rightarrow \hat{e}(\mathbf{P})$ forgetting the factor $N_{\mathbf{P}}$ in the Borel--Serre boundary component. We define the \textit{partial reductive Borel--Serre compactification} as a set
\begin{align*}
\XRBS=\coprod_{\mathbf{P}} \hat{e}(\mathbf{P}),
\end{align*}
where we once again interpret $\mathbf{G}$ as a parabolic subgroup with $\hat{e}(\mathbf{G})=X_G=X$.

\medskip

The projections $e(\mathbf{P})\rightarrow \hat{e}(\mathbf{P})$ define a surjection $\XBS\rightarrow \XRBS$ and we equip $\XRBS$ with the quotient topology. The action of $\mathbf{G}(\Q)$ on $\XBS$ descends to a continuous action on $\XRBS$ and the quotient $\Gamma\backslash \XRBS$ is a compact Hausdorff space (\cite[Lemma 2.4]{JiMurtySaperScherk}), and this is the reductive Borel--Serre compactification.

\begin{remark}
It should be noted here that the quotient topology on $\XRBS$ does \textit{not} agree with that of the uniform construction in \cite{BorelJi} which is much weaker. We believe that the uniform construction agrees with the weaker topology $(\XRBS)^w$ of \Cref{remark about weak topology} with respect to the action of $\Gamma$.
\exend
\end{remark}

We have a commutative diagram of quotient maps as on the left below, which, if $\Gamma$ is neat, restricts to a commutative diagram of fibre bundles as on the right for each rational parabolic subgroup $\mathbf{P}$, where $\Gamma_{\mathbf{P}}=\Gamma\cap \mathbf{P}(\Q)$ and $\Gamma_{\mathbf{L_P}}=\Gamma_{\mathbf{P}}/\Gamma_{\mathbf{N_P}}$ with $\Gamma_{\mathbf{N_P}}=\Gamma\cap \mathbf{N_P}(\Q)$. The fibre of the lower horizontal map is the nilmanifold $\Gamma_{\mathbf{N_P}}\backslash N_{\mathbf{P}}$. Zucker originally defined the reductive Borel--Serre compactification as the quotient of $\Gamma\backslash \XBS$ given by collapsing these nilmanifold fibres.
\begin{center}
\begin{tikzpicture}
\matrix (m) [matrix of math nodes,row sep=2em,column sep=2em]
  {
     \XBS & \XRBS & & e(\mathbf{P}) & \hat{e}(\mathbf{P}) \\
     \Gamma\backslash \XBS & \Gamma\backslash \XRBS & & \Gamma_{\mathbf{P}}\backslash e(\mathbf{P}) & \Gamma_{\mathbf{L_P}}\backslash \hat{e}(\mathbf{P}) \\
  };
  \path[-stealth]
	(m-1-1) edge (m-1-2) edge (m-2-1)	
	(m-1-2) edge (m-2-2)
	(m-2-1) edge (m-2-2)
	(m-1-4) edge (m-1-5) edge (m-2-4)
	(m-2-4) edge (m-2-5)
	(m-1-5) edge (m-2-5)
  ;
\end{tikzpicture}
\end{center}

We will need the following observation: the spaces $\XBS$, $\Gamma\backslash \XBS$ and $\Gamma\backslash \XRBS$ are metrisable. Indeed, the partial Borel--Serre compactification is a second-countable manifold with corners (\cite[Theorem 7.8]{BorelSerre}), the Borel--Serre compactification is a compact manifold with corners, and the reductive Borel--Serre compactification is compact Hausdorff and locally metrisable.

\subsection{Stratifications and exit path \texorpdfstring{$\infty$}{infinity}-categories}\label{exit category of rbs}

In this section we use the results of \Cref{stratified spaces and group actions} to determine the exit path $\infty$-categories of the Borel--Serre and reductive Borel--Serre compactifications.

\medskip

We stratify the partial compactifications over the poset $\mathscr{P}$ of rational parabolic subgroups in the obvious way, sending the boundary component $e(\mathbf{P})$ respectively $\hat{e}(\mathbf{P})$ to $\mathbf{P}$. For the partial Borel--Serre compactification, this is the natural stratification of $\XBS$ as a manifold with corners; the codimension of the boundary component $e(\mathbf{P})$ is equal to the $\Q$-rank of $\mathbf{P}$: $\operatorname{rk}_{\Q}\mathbf{P}=\dim A_{\mathbf{P}}=|\Delta_{\mathbf{P}}|$. The action of $\Gamma$ on $\XBS$ and $\XRBS$ is a stratum preserving continuous action with $\Gamma$ acting on $\mathscr{P}$ by conjugation. Since a parabolic subgroup is its own normaliser, the stabiliser of $\mathbf{P}$ is $\Gamma_{\mathbf{P}}=\Gamma\cap \mathbf{P}(\Q)$ in both cases.

\medskip

Assume $\Gamma$ to be torsion free. The action of $\Gamma$ on $\XBS$ is free and properly discontinuous, the strata $e(\mathbf{P})$ are contractible and locally contractible, and, being a manifold with corners, $\XBS$ is metrisably conically stratified and the link spaces have contractible strata. Therefore \Cref{exit path categories and group actions with trivial stabilisers} applies.

\begin{theorem}\label{exit path category of the borel-serre compactification}
The exit path $\infty$-categories of the partial Borel--Serre compactification $\XBS$ and the Borel--Serre compactification $\Gamma\backslash \XBS$ are equivalent to the nerves of their homotopy categories. The homotopy category $\exit_1(\XBS)$ is in turn equivalent to the poset $\mathscr{P}$, and $\exit_1(\Gamma\backslash \XBS)$ is equivalent to the category $\mathscr{C}^{BS}_{\Gamma}$ with objects the rational parabolic subgroups of $\mathbf{G}$ and hom-sets
\begin{align*}
\mathscr{C}^{BS}_{\Gamma}(\mathbf{P},\mathbf{Q})=\{\gamma\in \Gamma\mid {}^\gamma \mathbf{P}\leq \mathbf{Q}\},\quad\text{for all}\quad\mathbf{P},\mathbf{Q}\in \mathscr{P},
\end{align*}
and composition given by the product in $\Gamma$.
\end{theorem}

Identifying the exit path $\infty$-category of the reductive Borel--Serre compactification requires a little more work in order to determine appropriate conical neighbourhoods. Moreover, the partial reductive Borel--Serre compactification is ``topologically conically stratified'' on non-compact link spaces, so the exit path simplicial set is not necessarily an $\infty$-category. We exploit that the paths in $\XBS$ descend to define a compatible collection paths in $\XRBS$ and apply \Cref{exit path categories and group actions II with composition}. We now assume $\Gamma$ to be neat.

\begin{theorem}\label{exit path category of the reductive borel-serre compactification}
The exit path $\infty$-category of the reductive Borel--Serre compactification $\Gamma\backslash \XRBS$ is equivalent to the nerve of its homotopy category. The homotopy category $\exit_1(\Gamma\backslash \XRBS)$ is in turn equivalent to the category $\mathscr{C}^{RBS}_{\Gamma}$ with objects the rational parabolic subgroups of $\mathbf{G}$ and hom-sets
\begin{align*}
\mathscr{C}^{RBS}_{\Gamma}(\mathbf{P},\mathbf{Q})=\{\gamma\in \Gamma\mid {}^\gamma\mathbf{P}\leq \mathbf{Q}\}/\Gamma_{\mathbf{N_P}},\quad\text{for all}\quad\mathbf{P},\mathbf{Q}\in \mathscr{P},
\end{align*}
where $\Gamma_{\mathbf{N_P}}$ acts by right multiplication, and composition is given by multiplication of representatives.
\end{theorem}
\begin{proof}
We will show that the stratified space $\XRBS\rightarrow \mathscr{P}$ equipped with the action of $\Gamma$ satisfies the conditions of \Cref{exit path categories and group actions II with composition}. Note first of all that the subgroup $\Gamma_{\mathbf{P}}^\ell\leq \Gamma$ which fixes the $\mathbf{P}$'th stratum $\hat{e}(\mathbf{P})$ pointwise is $\Gamma_{\mathbf{P}}^\ell=\Gamma_{\mathbf{N_P}}$.

\medskip

For all $\mathbf{P}\in \mathscr{P}$, stratify $\overline{A_{\mathbf{P}}}=(\R_{\geq 0})^{\Delta_{\mathbf{P}}}$ over $\mathscr{P}_{\geq \mathbf{P}}$ by sending $A_{\mathbf{P}}\cdot o_{\mathbf{Q}}$ to $\mathbf{Q}$ for all $\mathbf{Q}\geq \mathbf{P}$, where we recall that $A_{\mathbf{P}}\cdot o_{\mathbf{Q}}$ is the subspace spanned by the coordinates $\Delta_{\mathbf{P}}^{\mathbf{Q}}\subset \Delta_{\mathbf{P}}$ (all other coordinates are zero). These stratifications are compatible with the stratification of $\XBS$ as we have identified $e(\mathbf{Q})$ with $(A_{\mathbf{P}}\cdot o_{\mathbf{Q}})\times N_{\mathbf{P}}\times X_{\mathbf{P}}\subseteq X(\mathbf{P})$. For any $t>0$ and any rational parabolic subgroup $\mathbf{P}$, set
\begin{align*}
\overline{A_{\mathbf{P}}}(t):=\{a\in \overline{A_{\mathbf{P}}} \mid a^{-\alpha}<t\text{ for all }\alpha\in \Delta_{\mathbf{P}}\}\subseteq \overline{A_{\mathbf{P}}}
\end{align*}
stratified as a subspace of $\overline{A_{\mathbf{P}}}$, i.e. $\overline{A_{\mathbf{P}}}(t) = [0,t)^{\Delta_{\mathbf{P}}}$.

\medskip

Let $\mathbf{P}$ be a rational parabolic subgroup of $\mathbf{G}$. The group $\Gamma_{\mathbf{L_P}}$ is torsion free, as $\Gamma$ is neat, so it acts freely and properly discontinuously on the stratum $\hat{e}(\mathbf{P})$, Hence, we may choose an open, relatively compact and contractible subset $W\subseteq \hat{e}(\mathbf{P})$ such that
\begin{align*}
\{ \gamma\in \Gamma_{\mathbf{P}}\mid \gamma.W\cap W\neq \emptyset\}=\Gamma_{\mathbf{N_P}},
\end{align*}
and we can view $W$ as a subspace of the quotient $\Gamma_{\mathbf{L_P}}\backslash \hat{e}(\mathbf{P})$. Having compact fibres, the fibre bundle $\Gamma_{\mathbf{P}}\backslash e(\mathbf{P}) \rightarrow \Gamma_{\mathbf{L_P}}\backslash \hat{e}(\mathbf{P})$ is proper, and therefore the preimage $V\subseteq \Gamma_{\mathbf{P}}\backslash e(\mathbf{P})$ of $W$ under this map is relatively compact. The preimage of $V$ in $e(\mathbf{P})$ is $N_{\mathbf{P}}\times W$, and it follows that there is a $t>0$ such that the equivalence relations induced by $\Gamma$ and $\Gamma_{\mathbf{P}}$ on the subspace
\begin{align*}
\overline{A_{\mathbf{P}}}(t)\times N_{\mathbf{P}}\times W \subseteq \overline{A_{\mathbf{P}}}(t)\times e(\mathbf{P})\subseteq X(\mathbf{P})\subseteq \XBS
\end{align*}
of the partial Borel--Serre compactification agree (\cite[1.5]{Zucker86}).

\medskip

Consider the subspace
\begin{align*}
\ell_{\mathbf{P}}=\{a\in \overline{A_{\mathbf{P}}}\mid \textstyle\sum_{\alpha\in \Delta_{\mathbf{P}}} a^{-\alpha}=1\}\subseteq \overline{A_{\mathbf{P}}}
\end{align*}
and stratify it accordingly over $\mathscr{P}_{>\mathbf{P}}$. This is just the standard (topological) $(\dim A_{\mathbf{P}}-1)$-simplex embedded in the usual way in $\R^{\dim A_{\mathbf{P}}}=\R^{\Delta_{\mathbf{P}}}$ and stratified as a manifold with corners; the $\mathbf{Q}$'th stratum of $\ell_{\mathbf{P}}$ is $\ell_{\mathbf{PQ}}=\ell_{\mathbf{P}}\cap (A_{\mathbf{P}}\cdot o_{\mathbf{Q}})$. Let $C(\ell_{\mathbf{P}})\rightarrow \mathscr{P}_{\geq \mathbf{P}}$ denote the stratified cone on $\ell_{\mathbf{P}}$ (as $\ell_{\mathbf{P}}$ is compact Hausdorff, this agrees with the stratified topological cone $C^t(\ell_{\mathbf{P}})$). There is a stratum preserving embedding
\begin{align*}
C(\ell_{\mathbf{P}})\rightarrow \overline{A_{\mathbf{P}}}(t)
\end{align*}
given by sending $[a,s]\in C(\ell_{\mathbf{P}})$ to the point $b\in \overline{A_{\mathbf{P}}}(t)$ satisfying $b^{-\alpha}=st a^{-\alpha}$.

\medskip

Define a stratified space $\mathscr{L}_{\mathbf{P}}\rightarrow \mathscr{P}_{>\mathbf{P}}$ as the quotient of $\ell_{\mathbf{P}}\times N_{\mathbf{P}}\rightarrow \mathscr{P}_{>\mathbf{P}}$ given by applying the quotients 
\begin{align*}
\ell_{\mathbf{PQ}}\times N_{\mathbf{P}}\rightarrow \ell_{\mathbf{PQ}}\times N_{\mathbf{Q}}\backslash N_{\mathbf{P}}
\end{align*}
to the strata of $\ell_{\mathbf{P}}\times N_{\mathbf{P}}$. The embedding 
\begin{align*}
C(\ell_{\mathbf{P}})\times N_{\mathbf{P}}\times W\hookrightarrow \overline{A_{\mathbf{P}}}(t)\times N_{\mathbf{P}}\times W\hookrightarrow \XBS
\end{align*}
descends to define a stratum preserving embedding
\begin{align*}
C^t(\mathscr{L}_{\mathbf{P}})\times W\hookrightarrow \XRBS
\end{align*}
which restricts to the identity on $\{*\}\times W$ where $*$ is the apex of $C^t(\mathscr{L}_{\mathbf{P}})$ --- note that as $\mathscr{L}_{\mathbf{P}}$ is non-compact, the stratified topological cone is different from the stratified cone $C(\mathscr{L}_{\mathbf{P}})$. Let $U$ denote the image of the above map, so that we have a stratum preserving homeomorphism
\begin{align*}
\phi\colon C^t(\mathscr{L}_{\mathbf{P}})\times W\xrightarrow{\ \cong \ } U.
\end{align*}
Our choice of $W$ and $t$ imply that 
\begin{align*}
\{\gamma\in \Gamma\mid \gamma.U\cap U\neq \emptyset\}=\Gamma_{\mathbf{N_P}}.
\end{align*}
Moreover, $\phi$ is $\Gamma_{\mathbf{N_P}}$-equivariant, when we let $\Gamma_{\mathbf{N_P}}$ act on $\mathscr{L}_{\mathbf{P}}$ by acting on the second factor of the $\mathbf{Q}$'th stratum $\ell_{\mathbf{PQ}}\times N_{\mathbf{Q}}\backslash N_{\mathbf{P}}$ via the quotient $\Gamma_{\mathbf{N_Q}}\backslash \Gamma_{\mathbf{N_P}}$ for all $\mathbf{Q}$. The quotient $\Gamma_{\mathbf{N_P}}\backslash\mathscr{L}_{\mathbf{P}}$ is compact as it factors through $\ell_{\mathbf{P}}\times \Gamma_{\mathbf{N_P}}\backslash N_{\mathbf{P}}$. Hence, the conditions of \Cref{exit path categories and group actions II} are satisfied (the metrisability conditions are satisfied as $\Gamma\backslash \XRBS$ is metrisable).
 
\medskip

For a collection of compatible exit paths, we may choose the ones coming from the partial Borel--Serre compactification.  If $x,y\in \XBS$ are connected by a morphism in $\exit_1(\XBS)\simeq \mathscr{P}$, then this morphism is unique and we choose an exit path $p_{x\rightarrow y}$ in $\XBS$ representing this morphism (for $x=y$, we choose the trivial loop). Let $\mu\colon \XBS\rightarrow \XRBS$ denote the quotient map, and fix basepoints $x_{\mathbf{P}}\in e(\mathbf{P})$ for all $\mathbf{P}$. For any $\mathbf{P},\mathbf{Q}$ and $\gamma\in \Gamma$ with ${}^\gamma\mathbf{P}\leq \mathbf{Q}$, we choose the path $\mu \circ p_{x_{\mathbf{P}}\rightarrow \gamma^{-1}.x_{\mathbf{Q}}}$ in $\XRBS$. Then \Cref{exit path categories and group actions II with composition} applies and we are done.
\end{proof}

\begin{remark}
The uniform construction of \cite{BorelJi} gives rise to a conically stratified space equipped with an action of $\Gamma$ whose quotient space agrees with $\Gamma\backslash \XRBS$. We believe that one can apply \Cref{exit path categories and group actions} to this space directly, but by using the Borel--Serre compactification to define a collection of compatible exit paths, we save ourselves the trouble of having to analyse this topology in detail.
\exend
\end{remark}

\begin{remark}
We wish to remark that the identification of neighbourhoods and link spaces in the proof of \Cref{exit path category of the reductive borel-serre compactification} make no claim to originality (see \cite{JiMurtySaperScherk}, \cite{GoreskyHarderMacPherson}, \cite{Zucker86}, \cite{BorelJi}). We just make a detailed analysis in order to verify the conditions of \Cref{exit path categories and group actions II}.
\exend
\end{remark}

\begin{observation}\label{compatible equivalences}
The equivalences
\begin{align*}
\mathscr{C}^{BS}_{\Gamma}\rightarrow \exit_1(\Gamma\backslash \XBS)\quad\text{and}\quad \mathscr{C}^{RBS}_{\Gamma}\rightarrow \exit_1(\Gamma\backslash \XRBS)
\end{align*}
of the theorems above can be defined compatibly as follows: for any rational parabolic subgroup $\mathbf{P}$, choose a basepoint $x_{\mathbf{P}}\in e(\mathbf{P})$ in the Borel--Serre boundary component --- note that it in fact suffices to make a choice of basepoint $x_0\in X$, i.e.~a choice of maximal compact subgroup $K\leq G$, as this gives canonical choices of basepoints in the boundary components corresponding to the maximal compact subgroups $K\cap P\leq P$ for varying $\mathbf{P}$. For any two points $x,x'\in \XBS$, if there is a morphism $x\rightarrow x'$ in $\exit_1(\XBS)$, then it is unique, and we denote it by $p_{x\rightarrow x'}$.

\medskip

Recall the commutative diagram of quotient maps below. We denote by $(-)_*$ the induced map of exit path categories whenever this makes sense (the exit path simplicial set of the partial reductive Borel--Serre compactification is not necessarily an $\infty$-category).

\begin{center}
\begin{tikzpicture}
\matrix (m) [matrix of math nodes,row sep=2em,column sep=2em]
  {
     \XBS & \XRBS  \\
     \Gamma\backslash \XBS & \Gamma\backslash \XRBS \\
  };
  \path[-stealth]
	(m-1-1) edge node[above]{$\mu$} (m-1-2) edge node[left]{$\pi$} (m-2-1)	
	(m-1-2) edge node[right]{$\rho$} (m-2-2)
	(m-2-1) edge node[below]{$\nu$}(m-2-2)
  ;
\end{tikzpicture}
\end{center}
 
With respect to the basepoints $x_{\mathbf{P}}\in e(\mathbf{P})$, the equivalences
\begin{align*}
F^{BS}\colon \mathscr{C}^{BS}_{\Gamma}\rightarrow \exit_1(\Gamma\backslash \XBS),\quad\text{and}\quad F^{RBS}\colon\mathscr{C}^{RBS}_{\Gamma}\rightarrow \exit_1(\Gamma\backslash \XRBS)
\end{align*}
are given by
\begin{align*}
F^{BS}(\mathbf{P})=\pi(x_{\mathbf{P}})\quad\text{and}\quad F^{RBS}(\mathbf{P})=\rho(\mu(x_{\mathbf{P}})),
\end{align*}
on objects, and on morphisms by
\begin{align*}
&F^{BS}(\gamma\colon \mathbf{P}\rightarrow \mathbf{Q})=\pi_*(p_{x_{\mathbf{P}}\rightarrow \gamma^{-1}.x_{\mathbf{Q}}}),\\
&F^{RBS}([\gamma]\colon \mathbf{P}\rightarrow \mathbf{Q})=(\rho\circ \mu)_*(p_{x_{\mathbf{P}}\rightarrow \gamma^{-1}.x_{\mathbf{Q}}}).
\end{align*}

The following diagram commutes
\begin{center}
\begin{tikzpicture}
\matrix (m) [matrix of math nodes,row sep=2em,column sep=2em]
  {
     \mathscr{P} & \mathscr{C}^{BS}_{\Gamma} & \mathscr{C}^{RBS}_{\Gamma}  \\
     \exit_1(\XBS) & \exit_1(\Gamma\backslash \XBS) & \exit_1(\Gamma\backslash \XRBS) \\
  };
  \path[-stealth] 
	(m-1-1) edge (m-1-2) edge (m-2-1) 
	(m-1-2) edge  (m-1-3) edge node[right]{$F^{BS}$} (m-2-2) 
	(m-1-3) edge node[right]{$F^{RBS}$} (m-2-3)
	(m-2-1) edge node[above]{$\pi_*$} (m-2-2)
	(m-2-2) edge node[above]{$\nu_*$} (m-2-3)
  ;
\end{tikzpicture}
\end{center}
when $\mathscr{P}\rightarrow \mathscr{C}^{BS}_{\Gamma}$ is the inclusion as a subcategory sending the unique morphism $\mathbf{P}\leq \mathbf{Q}$ to the morphism $\mathbf{P}\rightarrow \mathbf{Q}$ given by the identity element in $\Gamma$; the functor $\mathscr{C}^{BS}_{\Gamma}\rightarrow \mathscr{C}^{RBS}_{\Gamma}$ is given by the obvious quotients on the hom-sets; and $\mathscr{P}\xrightarrow{\sim} \exit_1(\XBS)$ is given by sending $\mathbf{P}$ to $x_{\mathbf{P}}$.
\exend
\end{observation}

\section{Consequences: homotopy type and the constructible derived category}\label{consequences}

We derive some immediate corollaries to \Cref{exit path category of the reductive borel-serre compactification}, the identification of the exit path $\infty$-category of the reductive Borel--Serre compactification. We determine the homotopy type of the reductive Borel--Serre compactification and in particular the fundamental group, and we review the classification of constructible sheaves and the constructible derived category.

\medskip

Let $\mathbf{G}$ be a connected reductive linear algebraic group over $\Q$ of positive $\Q$-rank whose centre is anisotropic over $\Q$. For a given neat arithmetic group $\Gamma\leq \mathbf{G}(\Q)$, let $\Gamma\backslash \XRBS$ denote the reductive Borel--Serre compactification of the associated locally symmetric space $\Gamma\backslash X$ as defined in \Cref{rbs definition}. Let $\mathscr{C}^{RBS}_\Gamma$ be the category defined in \Cref{exit path category of the reductive borel-serre compactification}.

\medskip
 
Since the inclusion of the exit path $\infty$-category into the singular set is a weak homotopy equivalence of simplicial sets, we recover the homotopy type of the reductive Borel--Serre compactification (\Cref{homotopy type and group actions}).

\begin{corollary}\label{weak homotopy equivalence rbs with crbs}
The reductive Borel--Serre compactification $\Gamma\backslash \XRBS$ is weakly homotopy equivalent to the geometric realisation of $\mathscr{C}^{RBS}_{\Gamma}$.
\end{corollary}

The fundamental group of the geometric realisation of a small category is the localisation of the category at all morphisms (\cite[Proposition 1]{Quillen73}). We thus recover the following result of Ji--Murty--Saper--Scherk (\cite[Corollary 5.3]{JiMurtySaperScherk}).

\begin{corollary}\label{fundamental group of rbs}
The fundamental group of the reductive Borel--Serre compactification $\Gamma\backslash \XRBS$ is isomorphic to the group $\Gamma/E_\Gamma$, where $E_\Gamma\triangleleft \Gamma$ is the normal subgroup generated by the subgroups $\Gamma_{\mathbf{N_P}}\leq \Gamma$ as $\mathbf{P}$ runs through all rational parabolic subgroups of $\mathbf{G}$.
\end{corollary}

\begin{remark}
One should think of $E_\Gamma$ as the subgroup of ``elementary matrices'', cf. the case $\Gamma\leq \operatorname{GL}_n(\Z)$, $n\geq 3$.
\exend
\end{remark}

Having determined the exit path $\infty$-category, we get a classification of constructible sheaves on $\Gamma\backslash \XRBS$ as representations of $\mathscr{C}^{RBS}_\Gamma$ (\Cref{exit paths classify constructible sheaves}).

\begin{corollary}
For any compactly generated $\infty$-category $\mathscr{C}$, there is an equivalence of $\infty$-categories
\begin{align*}
\Psi_X\colon \operatorname{Fun}(N(\mathscr{C}^{RBS}_{\Gamma}), \mathscr{C})\rightarrow \operatorname{Shv}_{\operatorname{cbl}}(\Gamma\backslash \XRBS,\mathscr{C}).
\end{align*}
\end{corollary}

Now, as the exit path $\infty$-category of $\Gamma\backslash \XRBS$ is equivalent to the nerve of its homotopy category, we can apply \Cref{equivalence of derived infinity-categories} and \Cref{equivalence of derived 1-categories} to express the constructible derived category as a derived functor category.

\begin{theorem}\label{constructible derived category of RBS infinity}
Let $R$ be an associative ring. There is an equivalence of $\infty$-categories
\begin{align*}
\operatorname{Shv}_{\operatorname{cbl}}(\Gamma\backslash \XRBS,\operatorname{LMod}_R)\simeq \mathscr{D}(\operatorname{Fun}\big(\mathscr{C}^{RBS}_{\Gamma},\operatorname{LMod}_R^1)),
\end{align*}
which restricts to an equivalence
\begin{align*}
\operatorname{Shv}_{\operatorname{cbl,cpt}}(\Gamma\backslash \XRBS,\operatorname{LMod}_R)\simeq \mathscr{D}_{\operatorname{cpt}}(\operatorname{Fun}\big(\mathscr{C}^{RBS}_{\Gamma},\operatorname{LMod}_R^1)),
\end{align*}
where $\mathscr{D}_{\operatorname{cpt}}(\operatorname{Fun}\big(\mathscr{C}^{RBS}_{\Gamma},\operatorname{LMod}_R^1))\subset \mathscr{D}(\operatorname{Fun}\big(\mathscr{C}^{RBS}_{\Gamma},\operatorname{LMod}_R^1))$ is the full subcategory spanned by the complexes of functors $F_\bullet$ such that $F_\bullet(x)$ is a perfect complex for all $x\in X$.
\end{theorem}

The $1$-categorical version of this is as follows.

\begin{corollary}\label{constructible derived category of RBS}
Let $R$ be an associative ring. There is an equivalence of $1$-categories
\begin{align*}
D_{\operatorname{cbl}}(\operatorname{Shv}_1(\Gamma\backslash \XRBS,R))\simeq D(\operatorname{Fun}(\mathscr{C}^{RBS}_{\Gamma},\operatorname{LMod}_R^1))
\end{align*}
which restricts to an equivalence
\begin{align*}
D_{\operatorname{cbl,cpt}}(\operatorname{Shv}_1(\Gamma\backslash \XRBS,R))\simeq D_{\operatorname{cpt}}(\operatorname{Fun}(\mathscr{C}^{RBS}_{\Gamma},\operatorname{LMod}_R^1)),
\end{align*}
where $D_{\operatorname{cpt}}(\operatorname{Fun}(\mathscr{C}^{RBS}_{\Gamma},\operatorname{LMod}_R^1))\subset D(\operatorname{Fun}(\mathscr{C}^{RBS}_{\Gamma},\operatorname{LMod}_R^1))$ is the full subcategory spanned by the complexes of functors $F_\bullet$ such that $F_\bullet(x)$ is a perfect complex for all $x\in X$.
\end{corollary}

As mentioned earlier in \Cref{examples of constructible complexes of sheaves}, both intersection cohomology of $\Gamma\backslash \XRBS$ and weighted cohomology of $\Gamma$ are examples of constructible compact-valued complexes of sheaves on $\Gamma\backslash \XRBS$ taking values in complex vector spaces, i.e. they are objects of $D_{\operatorname{cbl,cpt}}(\Gamma\backslash \XRBS,\C)$ (\cite[\S 3]{GoreskyMacPherson83} and \cite[Theorem 17.6]{GoreskyHarderMacPherson}). In \cite{Saper05a} and \cite{Saper05b}, Saper introduced the theory of $\mathcal{L}$-modules, a combinatorial analogue of constructible complexes of sheaves on the reductive Borel--Serre compactification. The theory is used to prove a conjecture of Rapoport and Goresky--MacPherson relating the intersection cohomology of certain Satake compactifications with that of the reductive Borel--Serre compactification (\cite{Rapoport,GoreskyMacPherson88}). This allows one to transfer cohomological calculations from the more singular spaces, Satake compactifications, to the reductive Borel--Serre compactification.

\medskip

If one thinks of $\mathscr{L}$-modules as a combinatorial analogue of constructible complexes of sheaves, then the equivalence of \Cref{constructible derived category of RBS} can be interpreted as providing an actual combinatorial incarnation. The precise relationship between these two notions is unfortunately not completely evident from the tools and calculations at hand --- further investigation of this should make the classification of constructible sheaves more explicit and more accessible to possible applications.

\section{Groups acting on posets}\label{Groups acting on posets}

This section does not contain any novel results, but is included in order to give some perspective on the main results of the paper. The category $\mathscr{C}^{RBS}_{\Gamma}$ was defined in \Cref{exit category of rbs} in terms of the poset of rational parabolic subgroups of a reductive algebraic group, their unipotent radicals and the conjugation action of $\Gamma$ on this poset --- it is a special case of the category $\mathscr{C}_{G,X}$ defined in \Cref{exit path categories and group actions} in terms of stabilisers and poset relations for a group acting on a stratified space. The object of interest was the stratified space $\Gamma\backslash\XRBS$ or in the general case a stratified orbit space $G\backslash X$. It is easy to see, however, that these categories make sense in a more general setting of a group acting on a poset, and moreover, that there are well-known examples of these categories (at least their opposites) in the literature. We make this generalisation precise and provide concrete examples.

\subsection{Construction and examples}

We generalise the definition of $\mathscr{C}^{RBS}_{\Gamma}$ to group actions on posets and give several examples of such categories.

\begin{construction}\label{generalisation}
Let $G$ be a group acting on a poset $I$. Let $G_i$ denote the stabiliser of $i\in I$, and suppose we have a choice of subgroup $G_i^\ell\leq G_i$ for every $i\in I$ such that the following conditions hold:
\begin{enumerate}[label=(\roman*)]
\item $G_j^{\ell}\leq G_i^{\ell}$ for all $i\leq j$;
\item ${}^gG_i^{\ell}=G_{g.i}^\ell$ for all $i\in I$, $g\in G$.
\end{enumerate}
We call $G_i^\ell$ the \textit{link subgroup} at $i$. Define a category $\mathscr{C}_{G,I}$ with objects the elements of $I$ and hom-sets
\begin{align*}
\mathscr{C}(i,j)=\{g\in G\mid g.i\leq j\}/G_i^\ell,
\end{align*}
where $G_i^\ell$ acts by right multiplication, and with composition given by multiplication of representatives in $G$. Properties (i) and (ii) imply that this is well-defined.
\exend
\end{construction}

\begin{example}\label{example}
\ 
\begin{enumerate}[label=(\roman*)]
\item Let $X\rightarrow I$ be a Hausdorff stratified space such that $X_i\subset \overline{X_j}$ for all $i\leq j$. Suppose $G$ is a discrete group acting on $X\rightarrow I$. Let for all $i\in I$, $G_i^\ell\leq G_i$ denote the subgroup which fixes $X_i$ pointwise. This recovers the category $\mathscr{C}_{G,X}$ in the situations of \Cref{exit path categories and group actions} and \Cref{exit path categories and group actions II with composition}.

\item For any group $G$ and any collection of subgroups $\mathcal{C}$ which is closed under conjugation, we can view $\mathcal{C}$ as a poset and consider the action of $G$ on $\mathcal{C}$ by conjugation and choose the trivial subgroups as the link subgroups. This recovers the transport category on the collection $\mathcal{C}$.

\item Let $\mathbf{G}$ be a connected linear algebraic group defined over a field $k$ and let $\mathscr{P}$ denote the poset of $k$-parabolic subgroups of $\mathbf{G}$. The group $\mathbf{G}(k)$ acts on $\mathscr{P}$ by conjugation. Let for all $\mathbf{P}\in \mathscr{P}$, $\mathbf{N_P}\leq \mathbf{P}$ denote the unipotent radical and choose the $k$-points of these as the link subgroups: $(\mathbf{G}(k))_{\mathbf{P}}^\ell=\mathbf{N_P}(k)\leq \mathbf{P}(k)\leq (\mathbf{G}(k))_{\mathbf{P}}$, where $(\mathbf{G}(k))_{\mathbf{P}}$ is the normaliser of $\mathbf{P}$ in $\mathbf{G}(k)$.

\item As an extension of the previous example, we can also consider the action of a subgroup $\Gamma\leq \mathbf{G}(k)$ and the restricted subgroups $\Gamma_{\mathbf{P}}^\ell =\Gamma_{\mathbf{N_P}}=\Gamma \cap \mathbf{N_P}(k)$. In the situation of \Cref{rbs definition} and for $\Gamma$ a neat arithmetic group, this recovers the category $\mathscr{C}^{RBS}_{\Gamma}$ of \Cref{exit path category of the reductive borel-serre compactification}. If we choose the trivial subgroups $e\leq \Gamma_{\mathbf{P}}$ as the link subgroups, then we recover $\mathscr{C}^{BS}_{\Gamma}$ of \Cref{exit path category of the borel-serre compactification}. Note that these categories are also recovered in (i) above when considering the action of $\Gamma$ on the partial Borel--Serre respectively partial reductive Borel--Serre compactifications as done in \Cref{rbs}.

\item Let $G=(G,B,N,S,U)$ be a finite group with a split BN-pair of characteristic $p$ (see \cite[\S 69]{CurtisReiner} for details). Let $\mathscr{P}$ denote the collection of parabolic subgroups of $G$, i.e.~the subgroups $P$ containing some conjugate of $B$. Then $G_P=P$ for all $P\in \mathscr{P}$ (\cite[Theorem 65.19]{CurtisReiner}). As link subgroups, consider the maximal normal $p$-subgroups, $O_p(P)\leq P$ (the analogue of the unipotent radical). This recovers (the opposite of) the orbit category on the $p$-radical subgroups of $G$, an object of great interest in finite group theory (see \Cref{orbit categories and p-radical subgroups} below).

\item We can generalise (iii) and (iv) to the case of reductive group schemes: for a reductive group scheme $\mathbf{G}$ over a scheme $S$, consider the poset $\mathscr{P}$ of parabolic subgroups and for each $\mathbf{P}\in \mathscr{P}$, let $\mathbf{N_P}$ denote the unipotent radical of $\mathbf{P}$ (\cite[\S 5.2]{Conrad14}). Any subgroup $\Gamma\leq \mathbf{G}(S)$ acts on $\mathscr{P}$ by conjugation and we can choose the link subgroups $\Gamma_{\mathbf{P}}^\ell=\Gamma_{\mathbf{N_P}}=\Gamma\cap \mathbf{N_P}(S)$ given by the unipotent radicals. 

\item Let $S$ be a surface of finite type and let $\Gamma(S)$ denote the corresponding mapping class group. Consider the poset $\mathcal{P}(S)$ whose elements are isotopy classes of disjoint closed simple curves on $S$ (including the empty set) and whose partial order is given by refinement, i.e. reverse inclusion (this is the opposite of the augmented curve complex and it encodes the natural stratification of the augmented Teichmüller space (see for example \cite{HubbardKoch})). The mapping class group $\Gamma(S)$ acts on $\mathcal{P}(S)$ in the obvious way by applying the homeomorphisms of $S$ to the closed simple curves. Given an element $\sigma\in \mathcal{P}(S)$, let $\Gamma(S)_\sigma^\ell\subset \Gamma(S)_\sigma$ denote the subgroup generated by Dehn twists around the components of $\sigma$. The resulting category is (the opposite of) the Charney--Lee category $\mathcal{CL}(S)$ of \cite{CharneyLee84, EbertGiansiracusa, ChenLooijenga} and we refer to these sources for details. This category is also recovered in an orbifold version of (i) above by considering the action of the mapping class group on the augmented Teichmüller space whose quotient can be identified with the Deligne--Mumford compactification --- a suitable orbifold version of the results of this paper should identify the exit path $\infty$-category of the Deligne--Mumford compactification with (the opposite of) the Charney--Lee category.

\item Let $A$ be an associative ring and $M$ a finitely generated projective $A$-module, and let $\mathscr{F}$ denote the poset of splittable flags of submodules of $M$. The group $GL(M)$ acts on $\mathscr{F}$ by conjugation. For every flag $\mathcal{F}\in \mathscr{F}$, let $(GL(M))_{\mathcal{F}}^\ell$ denote the subgroup of elements preserving $\mathcal{F}$ which induce the identity on the associated graded of $\mathcal{F}$. The resulting category is the category $\operatorname{RBS}(M)$ introduced in \cite{ClausenOrsnesJansen}. If $A$ is commutative and $\operatorname{Spec}(A)$ is connected, then splittable flags correspond to parabolic subgroups of the reductive group scheme $GL(M)$ and the category $\operatorname{RBS}(M)$ coincides with the one in (vi) above (see the discussion in \cite[\S 5]{ClausenOrsnesJansen}). \exend
\end{enumerate}
\end{example}

\subsection{Orbit categories and \texorpdfstring{$p$}{p}-radical subgroups}\label{orbit categories and p-radical subgroups}

The categories defined in \Cref{generalisation} appear in the field of finite group theory as mentioned in \Cref{example} (v) above. We spell out the identification of the resulting category as the orbit category on p-radical subgroups to underline the fact that it appears both in a different setting and in a different incarnation. It is a simple application of the Borel--Tits Theorem.

\medskip

Let $G=(G,B,N,S,U)$ be a finite group with a split BN-pair, and consider the categories $\mathscr{C}^{BS}_G$ respectively $\mathscr{C}^{RBS}_G$ obtained from \Cref{generalisation} by considering the poset of parabolic subgroups of $G$ and as link subgroups the trivial subgroups $e\leq P$ respectively the largest normal $p$-groups $O_p(P)\leq P$ (cf. \Cref{example} (ii) respectively (v)). There is a canonical functor $\mathscr{C}^{BS}_G\rightarrow \mathscr{C}^{RBS}_G$ which is the identity on objects and is given on hom-sets by the quotient maps
\begin{align*}
\{g\in G\mid {}^g P\leq Q\} \longrightarrow \{g\in G\mid {}^gP\leq Q\}/O_p(P).
\end{align*}

\begin{remark}
These categories generalise the exit path categories of the Borel--Serre respectively reductive Borel--Serre compactifications and the functor generalises the one induced by the quotient map $\Gamma\backslash\XBS\rightarrow \Gamma \backslash\XRBS$ as found in \Cref{compatible equivalences}.
\exend
\end{remark}

\begin{definition}
Let $G$ be any finite group and $p$ a prime. A subgroup $U\leq G$ is called $p$-\textit{radical} if the greatest normal $p$-subgroup of the normaliser of $U$ in $G$ is $U$ itself, i.e.~if $O_p(N_G(U))=U$ or equivalently $O_p(N_G(U)/U)=e$.
\defend
\end{definition}

\begin{remark}
The $p$-radical subgroups have been studied extensively in finite group theory: they play an important role in Alperin's weight conjecture \cite{Alperin, AlperinFong} and the poset of $p$-radical subgroups and the orbit category on this collection turn out to be of great significance in group cohomology and homotopy theory of classifying spaces (\cite{Bouc,JackowskiMcClureOliverI,JackowskiMcClureOliverII,Grodal02,Grodal18}).
\exend
\end{remark}

For $G$ a finite group with a split BN-pair of characteristic $p$, let $\mathscr{O}(G)$ denote the orbit category of $G$-orbits and $G$-maps and denote by $\mathcal{B}_p^e(G)$ the collection of $p$-radical subgroups of $G$. Consider the transport category $\mathscr{T}_{\mathcal{B}_p^e(G)}(G)$ on the collection of $p$-radical subgroups of $G$, and the full subcategory $\mathscr{O}_{\mathcal{B}_p^e(G)}(G)\subseteq \mathscr{O}(G)$ spanned by the $G$-orbits whose isotropy group is a $p$-radical subgroup of $G$. There is a canonical functor
\begin{align*}
\mathscr{T}_{\mathcal{B}_p^e(G)}(G) \rightarrow \mathscr{O}_{\mathcal{B}_p^e(G)}(G)
\end{align*}
which sends $P$ to $G/O_p(P)$ and on hom-sets is given by inversion and taking quotients, $g\mapsto [g^{-1}]$:
\begin{align*}
\{g\in G\mid {}^g O_p(Q)\leq O_p(P)\} \rightarrow \{g\in G\mid O_p(Q)^g\leq O_p(P)\}/O_p(P),
\end{align*}
where we use that $\operatorname{Hom}_G(G/H,G/K)\xrightarrow{\ \cong\ }\{g\in G\mid H^g\leq K\}/K$ by sending a $G$-map to its value on the identity coset.

\medskip

The poset of parabolic subgroups of $G$ is $G$-equivalent to the (opposite) poset of $p$-radical subgroups of $G$ by taking normaliser and $O_p$ respectively. This is a well-known fact and a consequence of the Borel--Tits Theorem, which says that if a closed unipotent subgroup $U$ of a connected algebraic group $\mathbf{H}$ is equal to the unipotent radical of its normaliser, then $N_{\mathbf{H}}(U)$ is a parabolic subgroup of $\mathbf{H}$ (see \cite[Corollary 3.2]{BorelTits71} for the general case and \cite{BurgoyneWilliamson} for the analogous result for finite Chevalley groups). The following proposition is a simple application of this fact --- we spell out the steps for clarity (see also for example \cite[Remark 4.3]{Grodal02}).

\begin{proposition}
There is a commutative diagram
\begin{center}
\begin{tikzpicture}
\matrix (m) [matrix of math nodes,row sep=2em,column sep=2em]
  {
	\mathscr{C}^{BS}_G & \mathscr{C}^{RBS}_G \\
	 \mathscr{T}_{\mathcal{B}_p^e(G)}(G)^{op} & \mathscr{O}_{\mathcal{B}_p^e(G)}(G)^{op} \\
  };
  \path[-stealth]
	(m-1-1) edge (m-1-2) edge node[left]{$\Psi$} (m-2-1)
	(m-2-1) edge (m-2-2)
	(m-1-2) edge node[right]{$\Phi$}(m-2-2)
  ;
\end{tikzpicture}
\end{center}
where the horizontal functors are the canonical ones and the vertical ones are isomorphisms given by
\begin{align*}
&\Psi(P)=O_p(P)& &\Psi(g\colon P\rightarrow Q)=g^{-1}\colon O_p(Q)\rightarrow O_p(P), \\
&\Phi(P)=G/O_p(P)& &\Phi([g]\colon P\rightarrow Q)=[g]\colon G/O_p(Q)\rightarrow G/O_p(P).
\end{align*}
\end{proposition}
\begin{proof}
The functors $\Phi$ and $\Psi$ are well-defined as $N_G(O_p(P))=P$ for all parabolic subgroups $P$ (\cite[Theorem 69.10]{CurtisReiner}). They are bijective on objects by the Borel--Tits theorem. To see that they are bijective on hom-sets, note that
\begin{align}
\{g\in G \mid  O_p(Q)^g\leq O_p(P)\}=\{g\in G\mid {}^g P\leq Q\},
\end{align}
since in the case where $P\leq Q$, both sets are equal to $Q$ (this is seen in the proof of \cite[Lemma 4.2]{Grodal02} and also in \cite[Theorem 65.19]{CurtisReiner}).
\end{proof}

\appendix
                                                                                                                                                                                                                                                                                                     
\section{Homotopy links and fibrations}\label{appendixA}

We provide proofs of the two fundamental results on homotopy links used in \Cref{exitcat}.  These are elementary point-set topological proofs and the results are well-known. We include them for the sake of self-containment, and since the proofs that we have been able to locate in the literature work in much more general or slightly different settings. It also clarifies why we impose metrisability conditions on the stratified spaces.

\medskip

Throughout this appendix, we write $I=[0,1]$ for the unit interval to ease notation. This should not be confused with the posets $I$ appearing in the main body of the paper.

\medskip

We recall the definition of the homotopy link, also given in \Cref{exitcat}: let $X$ be a topological space and $Y\subseteq X$ a subspace. The \textit{homotopy link} of $Y$ in $X$ is defined as the following subspace of paths
\begin{align*}
H(X,Y)=\{\gamma\colon I\rightarrow X\mid \gamma(0)\in Y,\ \gamma((0,1])\subseteq X-Y\}\subset C(I,X)
\end{align*}
equipped with the compact-open topology.

\medskip

This is a notion from the theory of homotopically stratified sets introduced by Quinn in \cite{Quinn88} in order to study purely topological stratified phenomena: a filtered space 
\begin{align*}
X_0\subset X_1 \subset \cdots \subset X_n
\end{align*}
is \textit{homotopically stratified} if for all $k>i$, the subspace $X_i-X_{i-1}$ has a ``homotopically well-behaved'' neighbourhood in $(X_k-X_{k-1})\cup (X_i-X_{i-1})$ (i.e.~is tame, see \Cref{tame remark}) and the evaluation at zero map from the homotopy link of this pair is a fibration. These conditions provide a homotopical replacement of mapping cylinder neighbourhoods, the homotopy link being an analogue of the frontier of such a mapping cylinder neighbourhood (see also \cite{Quinn02}).

\subsection{End point evaluation fibrations}

We show that for a suitably nice pair of spaces $(X,Y)$, the end point evaluation map $H(X,Y)\rightarrow X\times (X-Y)$ is a fibration. The following lemma explains our need to impose metrisability conditions on the stratified spaces that we consider.

\begin{lemma}\label{retricting to a path within N}
Let $X$ be a metrisable space, $Y\subseteq X$ a subspace, and $U$ an open neighbourhood of $Y$ in $X$. There is a continuous map $\delta\colon H(X,Y)\rightarrow (0,1)$ such that for all $\gamma\in H(X,Y)$, we have $\gamma([0,\delta(\gamma)])\subseteq U$ .
\end{lemma}
\begin{proof}
The homotopy link $H(X,Y)$ admits partitions of unity, being a subspace of a metrisable space $C(I,X)$ and thus itself metrisable. For any $\gamma\in H(X,Y)$, let $\delta_{\gamma}\in (0,1)$ such that $\gamma([0,\delta_{\gamma}])\subseteq U$. The subset
\begin{align*}
U_{\gamma}:=C([0,\delta_{\gamma}],U)\cap H(X,Y)=\{\eta\in H(X,Y)\mid \eta([0,\delta_{\gamma}])\subseteq U\}
\end{align*}
is an open neighbourhood of $\gamma$ in $H(X,Y)$. Let $\{\rho_{\gamma}\}$ be a partition of unity subordinate to the cover $\{U_{\gamma}\}$ of $H(X,Y)$, and define $\delta\colon H(X,Y)\rightarrow (0,1)$ as $\delta=\sum_\gamma \delta_\gamma\rho_\gamma$.
\end{proof}

\begin{proposition}\label{e_0 fibration top pair}
Let $X$ be a metrisable space, $Y\subseteq X$ a subspace, and suppose there is an open neighbourhood $N$ of $Y$ in $X$ such that the evaluation at zero map $H(N,Y)\rightarrow Y$ is a fibration. Then the evaluation at zero map $e_0\colon H(X,Y)\rightarrow Y$ is a fibration.
\end{proposition}
\begin{proof}
Let $A$ be a topological space, and let $\alpha_0$ and $\alpha$ as in the following diagram.
\begin{center}
\begin{tikzpicture}
\matrix (m) [matrix of math nodes,row sep=2em,column sep=2em]
  {
     A & H(X,Y) \\
     A\times I & Y \\
  };
  \path[-stealth]
	(m-1-1) edge node[above]{$\alpha_0$} (m-1-2)
	(m-2-1) edge node[below]{$\alpha$} (m-2-2)
	(m-1-1) edge (m-2-1)	
	(m-1-2) edge node[right]{$e_0$} (m-2-2)
  ;
  \path[dashed,-stealth]
  (m-2-1) edge node[above left]{$\tilde{\alpha}$} (m-1-2)
  ;
\end{tikzpicture}
\end{center}

Let $\delta\colon H(X,Y)\rightarrow (0,1)$ be as in \Cref{retricting to a path within N} for $U=N$. For any $\gamma\in H(X,Y)$ and any $0\leq r<s\leq 1$, let $\gamma_{[r,s]}\colon I\rightarrow X$ denote the reparametrisation of the restriction of $\gamma$ to $[r,s]$ and define continuous maps
\begin{align*}
&R\colon H(X,Y)\rightarrow H(N,Y), \qquad\phantom{mm}\gamma\mapsto \gamma_{[0,\delta(\gamma)]}, \\
&\overline{R}\colon H(X,Y)\rightarrow C(I,X-Y), \qquad \gamma\mapsto \gamma_{[\delta(\gamma),1]}. 
\end{align*}
By assumption, $e_0\colon H(N,Y)\rightarrow Y$ is a fibration, so there is a map $\hat{\alpha}\colon A\times I\rightarrow H(N,Y)$ such that $e_0\circ\hat{\alpha}=\alpha$ and $\hat{\alpha}(-,0)=R(\alpha_0(-))$.

\medskip

Consider the diagram below with $\eta$ given by
\begin{align*}
\eta(-,-,0)=\hat{\alpha}(-,-)(1),\quad\text{and}\quad \eta(-,0,-)=\overline{R}\circ\alpha_0,
\end{align*}
where $\overline{R}\circ \alpha_0\colon A\rightarrow C(I,X-Y)$ is viewed as a map $A\times I\rightarrow X-Y$ via the exponential law. The map $\hat{\eta}$ is an extension of $\eta$, using that the pair $(A\times I, A\times \{0\})$ has the homotopy extension property.

\begin{center}
\begin{tikzpicture}
\matrix (m) [matrix of math nodes,row sep=2em,column sep=2em]
  {
     A\times (I\times \{0\}\cup \{0\}\times I) & X-Y \\
     A\times I\times I & \\
  };
  \path[-stealth]
	(m-1-1) edge node[above]{$\eta$} (m-1-2) edge node[left]{$\iota$} (m-2-1)
  ;
  \path[dashed,-stealth]
  (m-2-1) edge node[below right]{$\hat{\eta}$} (m-1-2)
  ;
\end{tikzpicture}
\end{center}

We can view $\hat{\eta}$ as a map $A\times I\rightarrow C(I,X-Y)$ by applying the exponential law to the second factor of $I$, and we define $\tilde{\alpha}\colon A\times I\rightarrow H(X,Y)$ as the vertical concatenation of $\hat{\alpha}$ and $\hat{\eta}$:
\begin{align*}
\tilde{\alpha}(a,s)(t)=\begin{cases}
\hat{\alpha}(a,s)\big(\tfrac{t}{\delta(\alpha_0(a))}\big) & t\in [0,\,\delta(\alpha_0(a))] \\
\hat{\eta}(a,s)\bigg(\tfrac{t-\delta(\alpha_0(a))}{1-\delta(\alpha_0(a))}\bigg) &  t\in [\delta(\alpha_0(a)),\,1]
\end{cases}
\end{align*}

This is the desired lift.
\end{proof}

\begin{corollary}\label{e fibration top pair}
Let $X$ be a metrisable space, $Y\subseteq X$ a subspace, and suppose there is a neighbourhood $N$ of $Y$ in $X$ such that the evaluation at zero map $H(N,Y)\rightarrow Y$ is a fibration. Then the end point evaluation map $e=e_0\times e_1$ is a fibration:
\begin{align*}
e\colon H(X,Y)\longrightarrow Y\times (X-Y),\quad \gamma\mapsto (\gamma(0),\gamma(1)).
\end{align*}
\end{corollary}

We leave the proof of this as an exercise for someone wanting to practice concatenation and reparametrisation of homotopies. The strategy is as follows: given a path
\begin{align*}
(p_Y,p_{X-Y})\colon I\rightarrow Y\times (X-Y),
\end{align*}
we can apply \Cref{e_0 fibration top pair} and lift $p_Y$ along $e_0$ with specified starting point $p_0\in H(X,Y)$, resulting in a family $p_s\in H(X,Y)$, $s\in I$; we then concatenate each path $p_s$ with
\begin{itemize}
\item the path $t\mapsto p_{1-t}(1)$ restricted to $[1-s,1]$ (this runs back along the end points of the paths $p_t$ for $t\leq s$),
\item and the restriction of the given path $p_{X-Y}$ to $[0,s]$.
\end{itemize}
This strategy generalises to families of paths. See also \cite{Woolf} for more related fibrations (specifically Lemma 3.5).

\subsection{Homotopy links and mapping cylinder neighbourhoods}

We show that when we are only interested in the homotopical information, the homotopy link provides an adequate replacement for the link space or link bundle. For more details, see \cite{Quinn88}, in particular Lemma 2.4 and its corollary, or \cite{Quinn02}.

\begin{definition}\label{nearly strict deformation retraction}
Let $(N,Y)$ be a pair of spaces. A map $r\colon N\times I\rightarrow N$ is a \textit{nearly strict deformation retraction} into $Y$ if it satisfies:
\begin{enumerate}[label=(\roman*)]
\item $r(-,1)=\operatorname{id}$;
\item $r(N,0)\subseteq Y$;
\item $r(N-Y,t)\subseteq N-Y$ for all $t>0$;
\item $r(y,t)=y$ for all $y\in Y$, $t\in I$. \defend
\end{enumerate}

\end{definition}

\begin{remark}\label{tame remark}
The `nearly strict' refers to the fact that $r$ preserves the pair $(N,Y)$ until the very last moment at $t=0$, when everything is pushed into $Y$. In the setting of homotopically stratified sets, a subspace $Y\subseteq X$ is called tame if there exists a neighbourhood of $Y$ in $X$ equipped with a nearly strict deformation retraction (\cite{Quinn88}). 
\exend
\end{remark}

A nearly strict deformation retraction $r\colon N\times I\rightarrow N$ into a subspace $Y\subset N$ defines a continuous map $\Psi\colon N-Y \rightarrow H(N,Y)$, sending a point $x\in N-Y$ to the path $p_x\colon t\mapsto r(x,t)$ tracing the image of $x$ under $r$.

\begin{lemma}
Let $Y\subseteq N$ be a pair of topological spaces and suppose there is a nearly strict deformation retraction $r\colon N\times I\rightarrow N$ into $Y$. Then the map $\Psi\colon N-Y\rightarrow H(N,Y)$, $x\mapsto p_x$, defined above is a homotopy equivalence with homotopy inverse given by evaluation at $1$, $e_1\colon H(N,Y)\rightarrow N-Y$.
\end{lemma}
\begin{proof}
We first note that $e_1\circ \Psi=\operatorname{id}$. The map $H\colon H(N,Y)\times I\rightarrow H(N,Y)$ given by
\begin{align*}
H(\gamma,s)(t)=\begin{cases}
r(\gamma(t),1+2s(t-1)), & s\leq \tfrac{1}{2}\\
r(\gamma(2(t-ts+s)-1),t), & s\geq \tfrac{1}{2}
\end{cases}
\end{align*}
provides a homotopy $\operatorname{id} \sim \Psi \circ e_1$.
\end{proof}

\begin{lemma}
Let $X$ be a metrisable space, $Y\subseteq X$ a subspace, and $N$ an open neighbourhood of $Y$. The inclusion $\iota\colon H(N,Y)\rightarrow H(X,Y)$ is a homotopy equivalence.
\end{lemma}
\begin{proof}
Let $\delta\colon H(X,Y)\rightarrow (0,1)$ be as in \Cref{retricting to a path within N} for $U=N$, and define a map
\begin{align*}
G\colon H(X,Y)\times I\rightarrow H(X,Y),\qquad G(\gamma,s)(t)=\gamma(t(s\delta(\gamma) + 1-s))=\gamma_{[0,s\delta(\gamma) +1-s]}(t).
\end{align*}
Then $G_1=G(-,1)$, $\gamma\mapsto \gamma_{[0,\delta(\gamma)]}$, is a homotopy inverse to $\iota$ with $G$ providing the desired homotopies $G\colon \operatorname{id}_{H(X,Y)}\sim \iota\circ G_1$ and $G\circ \iota\colon \operatorname{id}_{H(N,Y)}\sim G_1\circ \iota$.
\end{proof}

Composing the homotopy equivalences from the above two lemmas, we have the following result.

\begin{proposition}\label{holink homotopy equivalent to neighbourhood}
Let $X$ be a metrisable space, and $Y\subseteq X$ a subspace. Suppose there is an open neighbourhood $N$ of $Y$ equipped with a nearly strict deformation retraction $r\colon N\times I \rightarrow N$. Then the map $N-Y\rightarrow H(X,Y)$, $x\mapsto p_x$, is a homotopy equivalence.
\end{proposition}

\section{Constructible \texorpdfstring{$\mathscr{C}$}{C}-valued sheaves}\label{appendix: sheaves}

In this appendix we show that the equivalence
\begin{align*}
\Psi_X\colon \operatorname{Fun}(\exit_\infty(X), \mathcal{S} )\xrightarrow{\simeq} \operatorname{Shv}_{\operatorname{cbl}}(X,\mathcal{S})
\end{align*}
of \cite[Theorem A.9.3]{LurieHA} can be generalised to $\mathscr{C}$-valued sheaves for compactly generated $\mathscr{C}$. Here $\mathcal{S}$ denotes the $\infty$-category of spaces. The equivalence is used in \Cref{constructible derived category} to give an expression of the constructible derived category of sheaves (of $R$-modules) in terms of the exit path $\infty$-category.

\medskip

To anyone with a reasonable grasp of $\infty$-categories, this will be quite rudimentary, but we hope that the level of detail will make the results more accessible to any reader without a background in $\infty$-categories.

\subsection{Sheaves valued in compactly generated \texorpdfstring{$\infty$}{infinity}-categories}\label{C-valued sheaves}

We give a very brief recap of the necessary definitions. We refer to \cite[\S 1.1]{LurieDAG-V} for details on $\mathscr{C}$-valued sheaves (see also \cite[\S 8.5]{Tanaka}). Let $\mathscr{C}$ be a compactly generated $\infty$-category. Then $\mathscr{C}\simeq \operatorname{Ind}(\mathscr{C}_0)$ for a small $\infty$-category $\mathscr{C}_0$ admitting small colimits (see comment at the beginning of \cite[\S 5.5.7]{LurieHTT}).

\begin{definition}
Let $X$ be a topological space and let $\mathscr{U}(X)$ denote the category of open sets of $X$. The $\infty$-category of \textit{$\mathscr{C}$-valued presheaves} on $X$ is the functor $\infty$-category
\begin{align*}
\operatorname{Fun}(N(\mathscr{U}(X))^{\operatorname{op}},\mathscr{C}).
\end{align*}
A presheaf $\mathcal{F}\colon N(\mathscr{U}(X))^{\operatorname{op}}\rightarrow \mathscr{C}$ is a \textit{$\mathscr{C}$-valued sheaf} on $X$ if for any $U\in \mathscr{U}(X)$ and any covering sieve $\{U_\alpha\}$ of $U$, the map
\begin{align*}
\mathcal{F}(U)\rightarrow \varprojlim\mathcal{F}(U_\alpha)
\end{align*}
is an equivalence. We denote the full subcategory of $\operatorname{Fun}(N(\mathscr{U}(X))^{\operatorname{op}},\mathscr{C})$ spanned by the $\mathscr{C}$-valued sheaves by $\operatorname{Shv}(X,\mathscr{C})$. 
\defend
\end{definition}

\begin{remark}
We say that a sheaf $\mathcal{F}\in \operatorname{Shv}(X,\mathscr{C})$ is hypercomplete if it satisfies descent with respect to any hypercovering not just covering sieves (see \cite[\S 6.5.3]{LurieHTT}).
\exend
\end{remark}

\begin{lemma}\label{swapping coefficients for sheaves}
For any topological space $X$ and any compactly generated $\mathscr{C}\simeq \operatorname{Ind}(\mathscr{C}_0)$, there is an equivalence of $\infty$-categories
\begin{align*}
\operatorname{Shv}(X, \mathscr{C})\xrightarrow{\simeq} \operatorname{Fun}^{\operatorname{lex}}(\mathscr{C}_0^{\operatorname{op}}, \operatorname{Shv}(X, \mathcal{S})),
\end{align*}
where the right-hand side, $\operatorname{Fun}^{\operatorname{lex}}(-,-)$, denotes the full subcategory spanned by the functors preserving finite limits.
\end{lemma}
\begin{proof}
For any $\infty$-category $\mathscr{D}$, there is an equivalence
\begin{align}\label{swapping coefficients}
\operatorname{Fun}(\mathscr{D}^{\operatorname{op}}, \operatorname{Ind}(\mathscr{C}_0))\xrightarrow{\simeq} \operatorname{Fun}^{\operatorname{lex}}(\mathscr{C}_0^{\operatorname{op}}, \operatorname{Fun}(\mathscr{D}^{\operatorname{op}}, \mathcal{S}));
\end{align}
this is in fact an isomorphism of simplicial sets identifying both sides with subcategories of $\operatorname{Fun}(\mathscr{D}^{\operatorname{op}}\times \mathscr{C}_0^{\operatorname{op}}, \mathcal{S})$. We apply this to (the nerve of) the category of open sets of $X$, $\mathscr{D}=N(\mathscr{U}(X))$, and note that by \cite[Corollary 5.1.2.3]{LurieHTT}, the sheaf condition on the left hand side of the equivalence translates to the sheaf condition on the codomain of the right hand side.
\end{proof}

\begin{remark}
Let $f\colon X\rightarrow Y$ be a map of topological spaces and consider the pushforward and pullback functors of $\mathcal{S}$-valued sheaves
\begin{align*}
\operatorname{Shv}(X,\mathcal{S})\mathop{\leftrightarrows}^{f^*}_{f_*} \operatorname{Shv}(Y,\mathcal{S}).
\end{align*}
Since both $f^*$ and $f_*$ preserve finite limits, postcomposition with these define an adjunction 
\begin{align*}
\operatorname{Fun}^{\operatorname{lex}}(\mathscr{C}_0^{\operatorname{op}}, \operatorname{Shv}(X, \mathcal{S}))\mathop{\leftrightarrows}^{f^*}_{f_*}\operatorname{Fun}^{\operatorname{lex}}(\mathscr{C}_0^{\operatorname{op}}, \operatorname{Shv}(Y, \mathcal{S})).
\end{align*}

Precomposition with the induced functor $\mathscr{U}(Y)\rightarrow \mathscr{U}(X)$ defines a pushforward map
\begin{align*}
f_*\colon \operatorname{Shv}(X, \mathscr{C})\rightarrow \operatorname{Shv}(Y, \mathscr{C}),
\end{align*}
and we have a commutative diagram as below.
\begin{center}
\begin{tikzpicture}
\matrix (m) [matrix of math nodes,row sep=2em,column sep=2em]
  {
	\operatorname{Shv}(X, \mathscr{C}) & \operatorname{Fun}^{\operatorname{lex}}(\mathscr{C}_0^{\operatorname{op}}, \operatorname{Shv}(X, \mathcal{S})) \\
	\operatorname{Shv}(Y, \mathscr{C}) & \operatorname{Fun}^{\operatorname{lex}}(\mathscr{C}_0^{\operatorname{op}}, \operatorname{Shv}(Y, \mathcal{S})) \\
  };
  \path[-stealth] 
	(m-1-1) edge node[above]{$\simeq$} (m-1-2) edge node[left]{$f_*$} (m-2-1)
	(m-2-1) edge node[below]{$\simeq$} (m-2-2)
	(m-1-2) edge node[right]{$f_*$}(m-2-2)
  ;
\end{tikzpicture}
\end{center}

Therefore the left adjoint $f^*$ on the right hand side defines a left adjoint $f^*$ to the pushforward map of $\mathscr{C}$-valued sheaves on the left hand side. See also \cite[Remark 1.1.8]{LurieDAG-V}.
\exend
\end{remark}

\begin{definition}
Let $X$ be a topological space and let $\rho\colon X\rightarrow \ast$ denote the unique map to a point. A sheaf $\mathcal{F}\in \operatorname{Shv}(X,\mathscr{C})$ is \textit{constant} if it is in the essential image of the pullback functor $\rho^*\colon \operatorname{Shv}(\ast,\mathscr{C})\rightarrow \operatorname{Shv}(X, \mathscr{C})$. A sheaf $\mathcal{F}\in \operatorname{Shv}(X,\mathscr{C})$ is \textit{locally constant} if there is an open cover $\{U_\alpha\}$ of $X$ such that the pullback of $\mathcal{F}$ to each $U_\alpha$ is constant.
\defend
\end{definition}

\begin{definition}
Let $X$ be an $I$-stratified space. A sheaf $\mathcal{F}\in \operatorname{Shv}(X,\mathscr{C})$ is \textit{constructible} if the restriction to each $X_i$, $i\in I$, is locally constant. We denote by $\operatorname{Shv}_{\operatorname{cbl}}(X,\mathscr{C})$ the full subcategory of constructible sheaves.
\defend
\end{definition}

Since the equivalence of \Cref{swapping coefficients for sheaves} commutes with pullback functors, it descends to an equivalence of the subcategories of constructible sheaves.

\begin{lemma}\label{swapping coefficients constructible sheaves}
For any $I$-stratified space $X$ and any compactly generated $\mathscr{C}\simeq \operatorname{Ind}(\mathscr{C}_0)$, there is an equivalence
\begin{align*}
\operatorname{Shv}_{\operatorname{cbl}}(X, \mathscr{C})\xrightarrow{\simeq} \operatorname{Fun}^{\operatorname{lex}}(\mathscr{C}_0^{\operatorname{op}}, \operatorname{Shv}_{\operatorname{cbl}}(X, \mathcal{S}))
\end{align*}
\end{lemma}

\begin{definition}\label{constructible compact-valued sheaf}
Let $X$ be an $I$-stratified space. We say that a constructible $\mathscr{C}$-valued sheaf is \textit{compact-valued} if its stalks are compact objects of $\mathscr{C}$. We denote by $\operatorname{Shv}_{\operatorname{cbl,cpt}}(X,\mathscr{C})$ the full subcategory of constructible compact-valued sheaves.
\defend
\end{definition}

\subsection{Exit path \texorpdfstring{$\infty$}{infinity}-categories and constructible \texorpdfstring{$\mathscr{C}$}{C}-valued sheaves}

Using \Cref{swapping coefficients constructible sheaves}, we can generalise Lurie's classification of space-valued constructible sheaves as representations of the exit path $\infty$-category (\cite[Theorem A.9.3]{LurieHA}) to sheaves taking values in compactly generated $\infty$-categories. The result is well-known and not hard to prove, but we have been unable to locate a proof in the literature. See \cite[\S 8.6]{Tanaka} for a sketch of how to generalise the proof of \cite[Theorem A.9.3]{LurieHA} to $\mathscr{C}$-valued sheaves  --- the proof that we present here is relies entirely on Lurie's result rather than generalising the proof of it.

\begin{theorem}\label{exit paths and constructible sheaves equivalence arbitrary coefficients}
Let $X$ be a conically $I$-stratified space which is paracompact and locally contractible, and where $I$ satisfies the ascending chain condition. Let $\mathscr{C}$ be a compactly generated $\infty$-category. Then there is an equivalence of $\infty$-categories
\begin{align*}
\Psi_X\colon \operatorname{Fun}(\exit_\infty(X), \mathscr{C} )\xrightarrow{\simeq} \operatorname{Shv}_{\operatorname{cbl}}(X,\mathscr{C}).
\end{align*}
\end{theorem}
\begin{proof}
Let $\mathscr{C}_0$ denote the $\infty$-category of compact objects of $\mathscr{C}$. We get a sequence of equivalences
\begin{align*}
\operatorname{Fun}(\exit_\infty(X), \operatorname{Ind}(\mathscr{C}_0)) &\simeq\operatorname{Fun}^{\operatorname{lex}}(\mathscr{C}_0^{\operatorname{op}}, \operatorname{Fun}(\exit_\infty(X), \mathcal{S})) \\
&\xrightarrow{\simeq}\operatorname{Fun}^{\operatorname{lex}}(\mathscr{C}_0^{\operatorname{op}}, \operatorname{Shv}_{\operatorname{cbl}}(X, \mathcal{S}))\simeq \operatorname{Shv}_{\operatorname{cbl}}(X, \operatorname{Ind}(\mathscr{C}_0))
\end{align*}
by applying (\ref{swapping coefficients}) from the proof of \Cref{swapping coefficients for sheaves} to $\mathscr{D}=\exit_\infty(X)$ and combining this with the equivalences of \Cref{swapping coefficients constructible sheaves} and \cite[Theorem A.9.3]{LurieHA}.
\end{proof}

We have the following naturality statement generalising \cite[Proposition A.9.16]{LurieHA}.

\begin{proposition}\label{naturality}
Let $X\rightarrow I$ and $Y\rightarrow J$ be paracompact, locally contractible conically stratified spaces with $J\subset I$ and where $I$ satisfies the ascending chain condition. Let $\mathscr{C}$ be a compactly generated $\infty$-category. For any stratum preserving map $f\colon Y\rightarrow X$ which on posets is given by the inclusion, there is an equivalence $\phi_{Y,X}\colon \Psi_Y \circ f^*\Rightarrow f^*\circ \Psi_X$ as in the diagram below, which in particular commutes up to homotopy.
\begin{center}
\begin{tikzpicture}
\matrix (m) [matrix of math nodes,row sep=2em,column sep=2em]
  {
	\operatorname{Fun}(\exit_\infty(X),\mathscr{C}) & \operatorname{Shv}_{\operatorname{cbl}}(X, \mathscr{C}) \\
	\operatorname{Fun}(\exit_\infty(Y),\mathscr{C}) & \operatorname{Shv}_{\operatorname{cbl}}(Y, \mathscr{C}) \\
  };
  \path[-stealth] 
	(m-1-1) edge node[above]{$\Psi_X$} (m-1-2) edge node[left]{$f^*$} (m-2-1)
	(m-2-1) edge node[below]{$\Psi_Y$} (m-2-2)
	(m-1-2) edge node[right]{$f^*$}(m-2-2)
  ;
  	\path[shorten >=0.7cm,shorten <=0.7cm,-implies]
	(m-2-1) edge[double equal sign distance] node[below right]{$\phi_{Y,X}$} (m-1-2)
  ;
\end{tikzpicture}
\end{center}
\end{proposition}
\begin{proof}
The equivalence of \cite[Theorem A.9.3]{LurieHA} is the composite of three equivalences, the first two of which are natural. For the third one, Proposition A.9.16 of \cite{LurieHA} provides the desired equivalence of functors for $\mathcal{S}$-valued sheaves: 
\begin{center}
\begin{tikzpicture}
\matrix (m) [matrix of math nodes,row sep=2em,column sep=2em]
  {
	\operatorname{Fun}(\exit_\infty(X),\mathcal{S}) & \operatorname{Shv}_{\operatorname{cbl}}(X, \mathcal{S}) \\
	\operatorname{Fun}(\exit_\infty(Y),\mathcal{S}) & \operatorname{Shv}_{\operatorname{cbl}}(Y, \mathcal{S}) \\
  };
  \path[-stealth] 
	(m-1-1) edge node[above]{$\Psi_X$} (m-1-2) edge node[left]{$f^*$} (m-2-1)
	(m-2-1) edge node[below]{$\Psi_Y$} (m-2-2)
	(m-1-2) edge node[right]{$f^*$}(m-2-2)
	;
	\path[shorten >=0.7cm,shorten <=0.7cm,-implies]
	(m-2-1) edge[double equal sign distance] (m-1-2)
  ;
\end{tikzpicture}
\end{center}
Applying $\operatorname{Fun}^{\operatorname{lex}}(\mathscr{C}_0^{\operatorname{op}},-)$ to this diagram and noting that the equivalence of \Cref{swapping coefficients constructible sheaves} commutes with pullbacks (by definition of the pullback functor) and the equivalence (\ref{swapping coefficients}) in the proof of \Cref{swapping coefficients for sheaves} is natural, we obtain the desired equivalence $\phi_{Y,X}\colon \Psi_Y \circ f^*\Rightarrow f^*\circ \Psi_X$ for $\mathscr{C}$-valued sheaves.
\end{proof}

\begin{corollary}\label{stalks vs values}
Let $X$ be a paracompact, locally contractible conically $I$-stratified space with $I$ satisfying the ascending chain condition, and let $\mathscr{C}$ be a compactly generated $\infty$-category. Let $F\colon \exit_\infty(X)\rightarrow \mathscr{C}$ and $\mathscr{F}:=\Psi_X(F)\in \operatorname{Shv}_{\operatorname{cbl}}(X,\mathscr{C})$. For all $x\in X$, there is an equivalence $\mathscr{F}_x\xrightarrow{\simeq}F(x)$ in $\mathscr{C}$.
\end{corollary}
\begin{proof}
For the one point space $\ast$, the equivalence $\Psi_\ast\colon \operatorname{Fun}(\exit_\infty(\ast),\mathcal{S}) \rightarrow \operatorname{Shv}(\ast, \mathcal{S})$ sends a Kan complex $Y$ to the Kan complex $\operatorname{Fun}(\ast, Y)\cong Y$ (see \cite[Construction A.9.2]{LurieHA}). It follows that
\begin{align*}
\Psi_\ast\colon \operatorname{Fun}(\exit_\infty(\ast),\mathscr{C}) \rightarrow \operatorname{Shv}(\ast, \mathscr{C})
\end{align*}
is equivalent to the identity on $\mathscr{C}$. Applying \Cref{naturality} to the map $x\colon \ast\rightarrow X$ sending $\ast$ to $x\in X$ provides an equivalence $\mathscr{F}_x\xrightarrow{\simeq}F(x)$ in $\mathscr{C}$.
\end{proof}

The following is an immediate consequence of \Cref{stalks vs values}.

\begin{corollary}\label{exit paths and constructible sheaves equivalence arbitrary coefficients with compact stalks}
Suppose $X$ is a conically $I$-stratified space which is paracompact and locally contractible, and where $I$ satisfies the ascending chain condition. Let $\mathscr{C}$ be a compactly generated $\infty$-category and let $\mathscr{C}_0$ denote the subcategory of compact objects. Then the equivalence of \Cref{exit paths and constructible sheaves equivalence arbitrary coefficients} restricts to an equivalence
\begin{align*}
\Psi_X\colon \operatorname{Fun}(\exit_\infty(X), \mathscr{C}_0)\xrightarrow{\simeq} \operatorname{Shv}_{\operatorname{cbl,cpt}}(X,\mathscr{C}).
\end{align*}
\end{corollary}

\bibliographystyle{alpha}
\bibliography{biblio}

\begin{thebibliography}{HTWW00}

\bibitem[AF90]{AlperinFong}
J.~L. Alperin and P.~Fong.
\newblock Weights for symmetric and general linear groups.
\newblock {\em J. Algebra}, 131(1):2--22, 1990.

\bibitem[Alp87]{Alperin}
J.~L. Alperin.
\newblock Weights for finite groups.
\newblock In {\em The {A}rcata {C}onference on {R}epresentations of {F}inite
  {G}roups ({A}rcata, {C}alif., 1986)}, volume~47 of {\em Proc. Sympos. Pure
  Math.}, pages 369--379. Amer. Math. Soc., Providence, RI, 1987.

\bibitem[BJ06]{BorelJi}
Armand Borel and Lizhen Ji.
\newblock {\em Compactifications of symmetric and locally symmetric spaces}.
\newblock Mathematics: Theory \& Applications. Birkh\"auser Boston, Inc.,
  Boston, MA, 2006.

\bibitem[Bor69]{Borel69}
Armand Borel.
\newblock {\em Introduction aux groupes arithm\'etiques}.
\newblock Publications de l'Institut de Math\'ematique de l'Universit\'e de
  Strasbourg, XV. Actualit\'es Scientifiques et Industrielles, No. 1341.
  Hermann, Paris, 1969.

\bibitem[Bor74]{Borel74}
Armand Borel.
\newblock Stable real cohomology of arithmetic groups.
\newblock {\em Ann. Sci. \'Ecole Norm. Sup. (4)}, 7:235--272 (1975), 1974.

\bibitem[Bou84]{Bouc}
Serge Bouc.
\newblock Homologie de certains ensembles ordonn\'{e}s.
\newblock {\em C. R. Acad. Sci. Paris S\'{e}r. I Math.}, 299(2):49--52, 1984.

\bibitem[BS73]{BorelSerre}
A.~Borel and J.-P. Serre.
\newblock Corners and arithmetic groups.
\newblock {\em Comment. Math. Helv.}, 48:436--491, 1973.
\newblock Avec un appendice: Arrondissement des vari{\'e}t{\'e}s {\`a} coins,
  par A. Douady et L. H{\'e}rault.

\bibitem[BT71]{BorelTits71}
A.~Borel and J.~Tits.
\newblock \'{E}l\'{e}ments unipotents et sous-groupes paraboliques de groupes
  r\'{e}ductifs. {I}.
\newblock {\em Invent. Math.}, 12:95--104, 1971.

\bibitem[BW76]{BurgoyneWilliamson}
N.~Burgoyne and C.~Williamson.
\newblock On a theorem of {B}orel and {T}its for finite {C}hevalley groups.
\newblock {\em Arch. Math. (Basel)}, 27(5):489--491, 1976.

\bibitem[CJ21]{ClausenOrsnesJansen}
Dustin Clausen and Mikala~{\O}rsnes Jansen.
\newblock The reductive {B}orel--{S}erre compactification as a model for
  unstable algebraic {K}-theory.
\newblock Preprint:
  \href{https://arxiv.org/pdf/2108.01924.pdf}{arXiv:2108.01924}, 2021.

\bibitem[CL84]{CharneyLee84}
Ruth Charney and Ronnie Lee.
\newblock Moduli space of stable curves from a homotopy viewpoint.
\newblock {\em J. Differential Geom.}, 20(1):185--235, 1984.

\bibitem[CL21]{ChenLooijenga}
Jiaming Chen and Eduard Looijenga.
\newblock The homotopy type of the {B}aily-{B}orel and allied
  compactifications.
\newblock {\em Homology Homotopy Appl.}, 23(2):95--119, 2021.

\bibitem[Con14]{Conrad14}
Brian Conrad.
\newblock Reductive group schemes.
\newblock In {\em Autour des sch\'{e}mas en groupes. {V}ol. {I}}, volume 42/43
  of {\em Panor. Synth\`eses}, pages 93--444. Soc. Math. France, Paris, 2014.

\bibitem[CR87]{CurtisReiner}
Charles~W. Curtis and Irving Reiner.
\newblock {\em Methods of representation theory. {V}ol. {II}}.
\newblock Pure and Applied Mathematics (New York). John Wiley \& Sons, Inc.,
  New York, 1987.
\newblock With applications to finite groups and orders, A Wiley-Interscience
  Publication.

\bibitem[EG08]{EbertGiansiracusa}
Johannes Ebert and Jeffrey Giansiracusa.
\newblock On the homotopy type of the {D}eligne-{M}umford compactification.
\newblock {\em Algebr. Geom. Topol.}, 8(4):2049--2062, 2008.

\bibitem[GHM94]{GoreskyHarderMacPherson}
M.~Goresky, G.~Harder, and R.~MacPherson.
\newblock Weighted cohomology.
\newblock {\em Invent. Math.}, 116(1-3):139--213, 1994.

\bibitem[GM80]{GoreskyMacPherson80}
Mark Goresky and Robert MacPherson.
\newblock Intersection homology theory.
\newblock {\em Topology}, 19(2):135--162, 1980.

\bibitem[GM83]{GoreskyMacPherson83}
Mark Goresky and Robert MacPherson.
\newblock Intersection homology. {II}.
\newblock {\em Invent. Math.}, 72(1):77--129, 1983.

\bibitem[GM88]{GoreskyMacPherson88}
Mark Goresky and Robert MacPherson.
\newblock Weighted cohomology of {S}atake compactifications.
\newblock Centre de recherches math{\'e}matiques, preprint \# 1593, 1988.

\bibitem[GM92]{GoreskyMacPherson92}
Mark Goresky and Robert MacPherson.
\newblock Lefschetz numbers of {H}ecke correspondences.
\newblock In {\em The zeta functions of {P}icard modular surfaces}, pages
  465--478. Univ. Montr\'{e}al, Montreal, QC, 1992.

\bibitem[GM03]{GoreskyMacPherson03}
Mark Goresky and Robert MacPherson.
\newblock The topological trace formula.
\newblock {\em J. Reine Angew. Math.}, 560:77--150, 2003.

\bibitem[Gro02]{Grodal02}
Jesper Grodal.
\newblock Higher limits via subgroup complexes.
\newblock {\em Ann. of Math. (2)}, 155(2):405--457, 2002.

\bibitem[Gro18]{Grodal18}
Jesper Grodal.
\newblock Endotrivial modules for finite groups via homotopy theory.
\newblock Preprint: \href{https://arxiv.org/abs/1608.00499}{arXiv:1608.00499},
  2018.

\bibitem[Hat02]{Hatcher}
Allen Hatcher.
\newblock {\em Algebraic topology}.
\newblock Cambridge University Press, Cambridge, 2002.

\bibitem[HK14]{HubbardKoch}
John~H. Hubbard and Sarah Koch.
\newblock An analytic construction of the {D}eligne-{M}umford compactification
  of the moduli space of curves.
\newblock {\em J. Differential Geom.}, 98(2):261--313, 2014.

\bibitem[HTWW00]{HughesTaylorWeinbergerWilliams}
Bruce Hughes, Laurence~R. Taylor, Shmuel Weinberger, and Bruce Williams.
\newblock Neighborhoods in stratified spaces with two strata.
\newblock {\em Topology}, 39(5):873--919, 2000.

\bibitem[JM02]{JiMacPherson}
L.~Ji and R.~MacPherson.
\newblock Geometry of compactifications of locally symmetric spaces.
\newblock {\em Ann. Inst. Fourier (Grenoble)}, 52(2):457--559, 2002.

\bibitem[JMO92a]{JackowskiMcClureOliverI}
Stefan Jackowski, James McClure, and Bob Oliver.
\newblock Homotopy classification of self-maps of {$BG$} via {$G$}-actions.
  {I}.
\newblock {\em Ann. of Math. (2)}, 135(1):183--226, 1992.

\bibitem[JMO92b]{JackowskiMcClureOliverII}
Stefan Jackowski, James McClure, and Bob Oliver.
\newblock Homotopy classification of self-maps of {$BG$} via {$G$}-actions.
  {II}.
\newblock {\em Ann. of Math. (2)}, 135(2):227--270, 1992.

\bibitem[JMSS15]{JiMurtySaperScherk}
Lizhen Ji, V.~Kumar Murty, Leslie Saper, and John Scherk.
\newblock The fundamental group of reductive {B}orel-{S}erre and {S}atake
  compactifications.
\newblock {\em Asian J. Math.}, 19(3):465--485, 2015.

\bibitem[Lur09]{LurieHTT}
Jacob Lurie.
\newblock {\em Higher topos theory}, volume 170 of {\em Annals of Mathematics
  Studies}.
\newblock Princeton University Press, Princeton, NJ, 2009.

\bibitem[Lur11]{LurieDAG-V}
Jacob Lurie.
\newblock {D}erived {A}lgebraic {G}eometry {V}: {S}tructured {S}paces.
\newblock
  \href{https://www.math.ias.edu/~lurie/papers/DAG-V.pdf}{www.math.ias.edu/~lurie/papers/DAG-V.pdf},
  2011.

\bibitem[Lur17]{LurieHA}
Jacob Lurie.
\newblock Higher algebra.
\newblock
  \href{https://www.math.ias.edu/~lurie/papers/HA.pdf}{www.math.ias.edu/~lurie/papers/HA.pdf},
  2017.

\bibitem[mat]{mathoverflow265557}
mathoverflow.
\newblock
  \href{https://mathoverflow.net/questions/265557/sheaves-of-complexes-and-complexes-of-sheaves}{mathoverflow.net/questions/265557}.

\bibitem[Mil56]{Milnor56}
John Milnor.
\newblock Construction of universal bundles. {II}.
\newblock {\em Ann. of Math. (2)}, 63:430--436, 1956.

\bibitem[Qui73a]{Quillen73a}
Daniel Quillen.
\newblock Finite generation of the groups {$K\sb{i}$} of rings of algebraic
  integers.
\newblock In {\em Algebraic {$K$}-theory, {I}: {H}igher {$K$}-theories ({P}roc.
  {C}onf., {B}attelle {M}emorial {I}nst., {S}eattle, {W}ash., 1972)}, pages
  179--198. Lecture Notes in Math., Vol. 341, 1973.

\bibitem[Qui73b]{Quillen73}
Daniel Quillen.
\newblock Higher algebraic {$K$}-theory. {I}.
\newblock In {\em Algebraic {$K$}-theory, {I}: {H}igher {$K$}-theories ({P}roc.
  {C}onf., {B}attelle {M}emorial {I}nst., {S}eattle, {W}ash., 1972)}, pages
  85--147. Lecture Notes in Math., Vol. 341, 1973.

\bibitem[Qui88]{Quinn88}
Frank Quinn.
\newblock Homotopically stratified sets.
\newblock {\em J. Amer. Math. Soc.}, 1(2):441--499, 1988.

\bibitem[Qui02]{Quinn02}
Frank Quinn.
\newblock Lectures on controlled topology: mapping cylinder neighborhoods.
\newblock 9:461--489, 2002.

\bibitem[Rap86]{Rapoport}
Michael Rapoport.
\newblock Letter to {B}orel.
\newblock 1986.

\bibitem[Sap05a]{Saper05a}
Leslie Saper.
\newblock {$\mathcal{L}$}-modules and micro-support.
\newblock Preprint: \href{https://arxiv.org/abs/math/0112251}{arXiv:0112251},
  2005.
\newblock To appear in Annals of Mathematics.

\bibitem[Sap05b]{Saper05b}
Leslie Saper.
\newblock {$\mathcal{L}$}-modules and the conjecture of {R}apoport and
  {G}oresky-{M}ac{P}herson.
\newblock Number 298, pages 319--334. 2005.
\newblock Automorphic forms. I.

\bibitem[Tan19]{Tanaka}
Hiro~Lee Tanaka.
\newblock Cyclic structure and broken cycles.
\newblock Preprint:
  \href{https://arxiv.org/pdf/1907.03301.pdf}{arXiv:1907.03301}, 2019.

\bibitem[Tre09]{Treumann}
David Treumann.
\newblock Exit paths and constructible stacks.
\newblock {\em Compos. Math.}, 145(6):1504--1532, 2009.

\bibitem[Woo09]{Woolf}
Jon Woolf.
\newblock The fundamental category of a stratified space.
\newblock {\em J. Homotopy Relat. Struct.}, 4(1):359--387, 2009.

\bibitem[Zuc86]{Zucker86}
Steven Zucker.
\newblock {$L_2$}-cohomology and intersection homology of locally symmetric
  varieties. {II}.
\newblock {\em Compositio Math.}, 59(3):339--398, 1986.

\bibitem[Zuc01]{Zucker01}
Steven Zucker.
\newblock On the reductive {B}orel-{S}erre compactification: {$L^p$}-cohomology
  of arithmetic groups (for large {$p$}).
\newblock {\em Amer. J. Math.}, 123(5):951--984, 2001.

\bibitem[Zuc83]{Zucker82}
Steven Zucker.
\newblock {$L\sb{2}$}-cohomology of warped products and arithmetic groups.
\newblock {\em Invent. Math.}, 70(2):169--218, 1982/83.

\end{thebibliography}

\end{document}